\RequirePackage{fix-cm}
\documentclass[smallextended]{svjour3}       
\smartqed  
\usepackage{graphicx}
\usepackage{tabularx} 
\usepackage{graphicx}
\usepackage{subfigure}
\usepackage{amsmath}  
\usepackage{multicol}
\usepackage{multirow}
\usepackage{diagbox}
\usepackage{setspace}
\usepackage{verbatim}
\usepackage{amsfonts,amssymb}
\usepackage{graphicx} 
\usepackage[margin=1in,letterpaper]{geometry} 
\usepackage{algorithm}
\usepackage{comment}
\usepackage{algorithmic}
\usepackage{cite} 
\usepackage[final]{hyperref} 
\hypersetup{
	colorlinks=true,       
	linkcolor=blue,        
	citecolor=blue,        
	filecolor=magenta,     
	urlcolor=blue         
}
\renewcommand\tablename{Table}
\renewcommand\figurename{Figure}

\usepackage{float}
\usepackage{caption}

\usepackage[T1]{fontenc}
\usepackage{array}
\usepackage{booktabs}

\usepackage{epsfig}

\DeclareMathOperator*{\argmin}{argmin}

\begin{document}

\title{Image Restoration Models with Optimal Transport and Total Variation Regularization
}


\author{Weijia Huang         \and
        Zhongyi Huang  \and
        Wenli Yang   \and
        Wei Zhu 
}


\institute{Weijia Huang \at
            Department of Mathematical Sciences, Tsinghua University, Beijing 100084, China \\
              \email{hwj20@mails.tsinghua.edu.cn}           
           \and
           Zhongyi Huang \at
            Department of Mathematical Sciences, Tsinghua University, Beijing 100084, China \\
           \email{zhongyih@mail.tsinghua.edu.cn} 
            \and
           Wenli Yang \at
           School of Mathematics, China University of Mining and Technology, Xuzhou 221116, China \\
           \email{yangwl19@cumt.edu.cn}
            \and
           Wei Zhu \at
           Department of Mathematics, University of Alabama, Tuscaloosa, AL 35487, USA \\
           \email{wzhu7@ua.edu}
}

\date{Received: date / Accepted: date}

\maketitle


\begin{abstract}
    In this paper, we propose image restoration models using optimal transport (OT) and total variation regularization.
    We present theoretical results of the proposed models based on the relations between the dual Lipschitz norm from OT and the $G$-norm introduced by Yves Meyer. 
    We design a numerical method based on the Primal-Dual Hybrid Gradient (PDHG) algorithm for the Wasserstain distance and the augmented Lagrangian method (ALM) 
    for the total variation, and the convergence analysis of the proposed numerical method is established. 
    We also consider replacing the total variation in our model by one of its modifications developed in \cite{zhu}, with the aim of suppressing the stair-casing effect and preserving image contrasts.
    Numerical experiments demonstrate the features of the proposed models.
     
\end{abstract}

\section{Introduction}
\label{intro}
    The image restoration problem is a fundamental problem in signal and image processing which aims to recover a true image from an observed noisy and blurred image.
Many variational models and regularizers have been developed to solve this inverse problem.
Total variation (TV) is one of the most classical regularizers for image denoising proposed by Rudin, Osher, and Fetemi \cite{1992Nonlinear}. Assume $f$ is a given grayscale image, the Rudin-Osher-Fatemi (ROF) model is formulated as a minimization problem:
\begin{equation}\label{rof}
    \min_{u} \frac{\alpha}{2} \int_{\Omega} (f-u)^2 + \int_\Omega |\nabla u|,
\end{equation}
where $\alpha > 0$ is a tuning parameter, and $\int_\Omega |\nabla u|$ is the TV regularizer.
Although the ROF model can effectively remove image noise while preserving objects' boundaries, it suffers from the staircase effect and the degradation of image contrast \cite{meyeroscillating}. 
To overcome these limitations, variational models with higher-order regularizers have been proposed, including the total generalized variation \cite{Bredies2010TotalGV}, Euler's elastica \cite{Chan2002EULER,Euler}, the mean curvature \cite{zhu2012image,2020zhuLp}, etc.
These models suppress the staircase effect, although their associated minimization problems tend to be more complex than that of the ROF model. 
Besides the regularization term, the fidelity term in \eqref{rof}, given by the $L^2$-norm of $u-f$, is well-suited for Gaussian noise denoising. 
However, the real-world data set is corrupted by various types of noise. Therefore, models with different data fidelity terms have also been proposed to deal with specific noise types, including Rician noise \cite{2023A}, Poisson noise \cite{2017Augmented}, impulsive noise \cite{2011A}, etc.

Recently, optimal transport (OT) has been used in a variety of applications in image processing \cite{2017Optimal}, including image retrieval \cite{2000The}, image registration \cite{2004Optimal}, color transfer \cite{Chizat2018Scaling,2015Sliced}, computerized tomography reconstruction \cite{2014Iterative,2016Generalized}, etc. A survey of the application of OT in imaging problems can be found in \cite{2015Optimal}.

Burger, Franek and Sch{\"o}nlieb \cite{2012Regularized} proposed to use the Wasserstein metric for estimating and smoothing probability densities,
\begin{equation}\label{Reg}
    \min_u \frac{\lambda}{2} W_2(u, f)^2 + E(u),
\end{equation}
where $\lambda>0$ and $f \in \mathbb{P}(\Omega)$, the set of probability measure on $\Omega$ and $u$ is an estimated probability density. 
The regularizers $E(u)$ can be taken as total variation, and the data fidelity $W_2(u,\nu)^2$ is the Wasserstein-2 distance. 

The Wasserstein distance is a class of metrics to measure the distance between two probability measures or distributions with equal mass. 
For two probability measures $\mu,\nu$, the classic Kantorovitch formulation of the $L^p (p \geq 1)$ Wasserstein distance is defined as
    \begin{equation}\label{Kan}
        W_p(\mu, \nu) = \left\{ \inf_{\gamma \in \mathcal{U}(\mu, \nu)}  \int_{\Omega \times \Omega} |x-y|^p d\gamma(x,y)  \right\}^{1/p},
    \end{equation} 
where $|x-y|$ is the Euclidean distance which defines the transport cost between two points $x,y \in \Omega$, and the mass conservation laws impose that $\gamma$ is in the set
    \begin{equation}
        \mathcal{U}(\mu, \nu) = \left\{ \gamma \in \mathrm{Prob}(\Omega \times \Omega) : \quad \mathrm{Proj}_1\gamma = \mu, \quad \mathrm{Proj}_2\gamma = \nu \right\},
    \end{equation}
where $\mathrm{Prob}(\Omega \times \Omega)$ denotes the joint probability distributions on the product space and $\mathrm{Proj}_i \gamma,i=1,2$ are the marginals of $\gamma$. 
This metric also makes sense if $\mu, \nu$ are not probability measures
but still non-negative and have equal mass as $\int_\Omega \mu = \int_\Omega \nu$.
The optimal transportation problem \eqref{Kan} seeks a transportation plan $\gamma$ to transform the marginal $\mu$ into $\nu$ with minimum cost \cite{2015Convolutional}.
The advantage of using the Wasserstein distance is the ability to obtain structural properties of a density and the employment of the Wasserstein distance as a fidelity term gives a mass conservation constraint \cite{2012Regularized} on the solution and thus is density preserving.

In \cite{2014Imaging}, Lellmann and Sch{\"o}nlieb proposed to use the Kantorovich-Rubinstein (KR) norm from optimal transport as a data fidelity in the image denoising model. 
The Wasserstein-1 distance has an equivalent expression:
\begin{equation}\label{KR}
    W_1(\mu,\nu) = \mathrm{sup}_{\phi} \left\{ \int_\Omega \phi(x) d(\mu(x)-\nu(x)) : \mathrm{Lip}(\phi) \leq 1 \right\},
\end{equation} 
which only depends on the difference $\mu-\nu$.
This metric defines the dual Lipschitz norm 
\begin{equation}\label{lip}
    \|\mu\|_{\mathrm{Lip}^*} = \sup_{\phi}\left\{\int_{\Omega} \phi d\mu: \mathrm{Lip}(\phi) \leq 1\right\}.
\end{equation}
The Lipschitz constant of a function $\phi$ is denoted as
\begin{equation}\label{Lipschitz}
    \mathrm{Lip}(\phi) = \sup\left\{\frac{|\phi(x) - \phi(y)|}{|x-y|} : (x,y) \in \Omega \times \Omega, x \neq y\right\}.
\end{equation} 
This norm \eqref{lip} is only defined for $\mu$ such that $\int_\Omega d\mu = 0$, otherwise the supremum of \eqref{lip} is unbounded.
This can be fixed by adding a bound on the value of function $\phi$, which leads to the KR norm:
\begin{equation}\label{KR-norm}
    \|\mu\|_{KR,\lambda} = \mathrm{sup}\left\{\int_{\Omega} \phi d\mu:|\phi| \leq \lambda_1, \mathrm{Lip}(\phi) \leq \lambda_2\right\},
\end{equation}
for given $\lambda = (\lambda_1,\lambda_2)\succ 0$.
The Kantorovich-Rubinstein-TV denoising, or KR-TV denoising for short, is defined as the minimization problem:
\begin{equation}\label{KR-pro}
    \min_{u} \{\mathrm{TV}(u) + \|u-f\|_{KR,\lambda}\}, 
\end{equation}
where $\mathrm{TV}(u)=\int_{\Omega} |\nabla u|$. 
In \cite{2014Imaging}, it also presented the dual formulation of \eqref{KR-norm}: 
\begin{equation}\label{KR-pro-dual}
    \|\mu\|_{KR,\lambda} = \inf_{m} \lambda_1 \|\mu - \mathrm{div}(m)\|_{L_1} + \lambda_2\||m|\|_{L_1}. 
\end{equation}

The model \eqref{KR-pro} performs well when decomposing an image into a cartoon and oscillatory component, which is due to the fact that the norm \eqref{KR-pro-dual} has connections with the $G$-norm, introduced by Yves Meyer \cite{meyeroscillating} for capturing oscillating patterns, defined as
\begin{equation}\label{gnorm}
    \|v\|_G = \inf \{ \||m|\|_{L^{\infty}}\ : v = \mathrm{div}(m),\quad m \in L^{\infty}(\Omega) \},
\end{equation}
where $|m| = \sqrt{m_1^2 + m_2^2}$. 
Meyer proposed the following image decomposition model \cite{meyeroscillating}
\begin{equation}\label{g-pro}
    \inf_{(u,v)\in BV\times G} \left\{ \int_\Omega |\nabla u| + \lambda \|v\|_G :\quad f=u+v \right\} ,
\end{equation}
where the component $u$ represents cartoon, and the component $v$ depicts texture or noise.

In this paper, we propose a new image restoration model that makes use of optimal transport and total variation. 
The proposed model 
combines the total variation 
with the dual Lipschitz norm defined by the Wasserstein-1 distance,  
which has connections with the $G$-norm for capturing oscillating patterns in cartoon and texture image decomposition.
In our model, as in ROF, the total variation is used to capture objects of large scales, while the Wasserstein-1 distance, as a fidelity term, helps to keep some fine details like textures as the $G$-norm. 
In this work, we offer the analytical study of the proposed model based on the link between our model and the functional proposed by Vese and Osher in \cite{vese2003modeling} to approximate Meyer’s model \eqref{g-pro}, as generalizations of the results in \cite{meyeroscillating,osher2003image,2005ImageBMO}.

To minimize the proposed functional, we design a numerical algorithm based on algorithms for the Wasserstein-1 distance and the ROF model. 
In fact, many efficient TV minimization algorithms have been proposed including the augmented Lagrangian method (ALM) \cite{2017Augmented}, the primal-dual algorithm of Chambolle and Pock \cite{2011A}, the fast shrinkage thresholding algorithm (FISTA) \cite{2009A}, Chambolle's projection algorithm \cite{2004cha}, etc.
The special type of Wasserstein-1 distance can be computed based on the Primal-Dual Hybrid Gradient (PDHG) algorithm in \cite{2016A,2021MULTILEVEL,2017Vector}.
Inspired by a numerical approach for the G-TV model \cite{2005Image}, we solve the proposed model by minimizing a convex functional alternately in each variable and establish the convergence analysis of the proposed algorithm.

Moreover, in this paper, we extend our model by employing a variant of total variation developed in \cite{zhu}, aiming to
suppress the staircasing effect and preserve image contrasts.


The contribution of this work is summarized as follows:
\begin{enumerate}
    \item We propose image restoration models that make use of optimal transport and total variation regularization, and provide the analytical study of the proposed models based on the relation between the dual Lipschitz norm and the $G$-norm proposed by Yves Meyer \cite{meyeroscillating}.

    \item We construct numerical methods to solve the proposed models, and establish the convergence analysis of the proposed algorithm.

    \item We present experimental results to demonstrate the features of the proposed models, and apply a modification of total variation \cite{zhu} as regularization in order to promote image contrasts and suppress staircasing effect for image denoising.
\end{enumerate}

The rest of this paper is organized as follows. 
In Sect.~\ref{sec:rel}, we review the theory of optimal transport and Meyer's $G$-TV model.
In Sect.~\ref{sec:1}, we present our image restoration models and then discuss the theoretical results of the proposed models.
In Sect.~\ref{sec:alg}, we develop a numerical procedure for the proposed models and present the convergence analysis of the proposed algorithms.
Numerical results are presented in Sect.~\ref{sec:2} by applying the proposed models for real images.
Finally, we draw some conclusions in Sect.~\ref{sec:conclusion}.



\section{Related work}\label{sec:rel}

    \subsection{Optimal Transport} 
The computation of the Kantorovitch problem \eqref{Kan} involves the solution of a linear program (LP) whose cost is prohibitive whenever the distribution's dimension $n$ becomes large \cite{cuturi2013sinkhorn}.
A variety of algorithms have been proposed \cite{Peyr2019Computational}. 
In \cite{cuturi2013sinkhorn}, Cuturi proposed to approximate the optimal transport problems by adding an entropic regularization term and the resulting problem can be computed through Sinkhorn fixed-point algorithm, which is generalized to unbalanced optimal transport \cite{Chizat2018Scaling}. An inexact proximal point method for OT problem (IPOT) was developed in \cite{pmlr-v115-xie20b} which converges to the exact Wasserstein distance.
The model \eqref{Reg} is a variational problem that involves the Wasserstein distances and TV regularization, which can be solved by a computational framework proposed by Cuturi and Peyer in \cite{2015A}.
Xie et al. introduced the random block coordinate descent (RBCD) methods \cite{xie2024randomized} to directly solve this large-scale LP problem.

Benamou and Brenier proposed a dynamical formulation \cite{2000A} for the Monge-Kantorovitch problem \cite{rachev2006mass} within a fluid mechanics framework:
\begin{equation}\label{flu}
    \begin{split}
        & W_p(\mu,\nu)^p = \mathrm{inf}_{\rho,v \in C^0} \int_\Omega \int_0^1 \rho(t,x)|v(t,x)|^p dt dx,\\
        & s.t.\quad C^0 = \left \{(\rho,v) : \partial_t \rho + \mathrm{div}_x(\rho v) = 0, \quad \rho(0,\cdot) = \mu, \quad \rho(1,\cdot) = \nu \right\}.       
    \end{split}
\end{equation}
Recently, Gangbo et al. \cite{2019Unnormalized} proposed an extension for this fluid mechanics approach, enabling optimal transport of unnormalized and unequal masses. Gau et al. \cite{Gao_2023} proposed a variational model based on dynamic optimal transport for computerized tomography reconstruction.

The Kantorovich dual problem of \eqref{Kan} is
\begin{equation}\label{kan-dual}
    W_p(\mu,\nu) = \left( \sup_{\phi,\psi} \left\{ \int_\Omega \phi d\mu + \int_\Omega \psi d\nu: \phi(x) + \psi(y) \leq |x-y|^p, \quad \forall x,y \in \Omega \right\} \right)^{1/p}.
\end{equation}
The cost is $c(x,y)=|x-y|^p$ and the c-transform \cite{Peyr2019Computational} is defined as: 
\begin{equation}\label{c-tra}
    \phi^c(y) = \inf_{x \in \Omega} c(x,y) - \phi(x), \forall y \in \Omega, \quad 
    \psi^c(x) = \inf_{y \in \Omega} c(x,y) - \psi(y), \forall x \in \Omega.
\end{equation} 
Using \eqref{c-tra}, \eqref{kan-dual} can be reformulated as a convex program over a single potential,
\begin{eqnarray}
    W_p(\mu,\nu)^p & = \sup_{\psi} \int_\Omega \psi^c(x) d\mu(x) + \int_\Omega \psi(y) d\nu(y) \label{kan-dual-c}\\
    & = \sup_{\phi} \int_\Omega \phi(x) d\mu(x) + \int_\Omega \phi^c(y) d\nu(y).
\end{eqnarray}

Starting from the formulation \eqref{kan-dual-c}, one can replace the couple $(\psi^c,\psi)$ by $(\psi^c, (\psi^c)^c)$. 
According to Proposition 6.1 in \cite{Peyr2019Computational}, the pair $(\psi^c, (\psi^c)^c)$ is equivalent to any pair $(\phi,-\phi)$ such that $\mathrm{Lip}(\phi) \leq 1$, which leads to the expression for the $W_1$ distance \eqref{KR}.
 
The global Lipschitz constraint in \eqref{KR} can also be made as a uniform bound on the gradient of $\phi$, 
\begin{equation}\label{KR-1}
     W_1(\mu, \nu) = \sup_{\phi} \left\{ \int_{x \in \Omega} \phi(x)d(\mu(x) - \nu(x)) : \||\nabla \phi|\|_{\infty} \leq 1 \right\},
\end{equation}
where the constraint $\||\nabla \phi|\|_{\infty} \leq 1$ means that $|\nabla \phi(x)|\leq 1$ for any $x$.
The dual problem of \eqref{KR-1} is the following minimization problem,
\begin{equation}\label{KR-dual}
     W_1(\mu, \nu) = \inf_{m} \left\{ \||m|\|_{L^1} : \mathrm{div} (m) = \mu - \nu \right\}.
\end{equation}

Recently, the Wasserstein-1 distance \eqref{KR} has been utilized as a loss function within the framework of generative adversarial networks (WGANs) \cite{2017Wasserstein}. 
A number of numerical methods have been proposed for this special class of OT.
In \cite{2016A}, Li et al. pointed out that \eqref{KR-dual} is very similar to some problems which have been solved in the fields of compressed sensing and image processing.
So they proposed that problems \eqref{KR-1} and \eqref{KR-dual} can be jointly solved by their Lagrangians, which can be discretized and solved by using a first-order primal-dual algorithm \cite{2011A} with very simple updates at each iteration.
Later on, the paper \cite{2018Solving} introduced the proximal PD method method to calculate the Wasserstein-1 distance and proved that the number of iterations in this method is independent of the grid size.
Subsequently, the paper \cite{2021MULTILEVEL} applied the cascadic multilevel method to speed up the aforementioned algorithms, which is especially effective for large-scale problems. 
Moreover, a numerical method based on the proximal point algorithm was proposed \cite{L2017Measuring} to compute the KR norm that is used to measure the misfit between seismograms in full waveform inversion.




    
    \subsection{Meyer's $G$-TV Model for Cartoon and Texture Decomposition}

In Meyer's model \eqref{g-pro}, the $G$ space is suggested for the oscillatory component $v = \mathrm{div}(m) \in G$, which denotes the Banach space, 
\begin{equation}\label{g-space}
    G = \{f : \ f = \mathrm{div}(m), \quad \exists m \in L^{\infty}(\Omega)\}.
\end{equation}

Motivated by the following approximation to the $L^{\infty}$ norm of $|m|=\sqrt{m_1^2 + m_2^2}$, for $m_1,m_2\in L^{\infty}(\Omega)$: 
\begin{equation}\label{sob}
    \|\sqrt{m_1^2 + m_2^2}\|_{L^{\infty}} = \lim\limits_{p\to \infty} \|\sqrt{m_1^2 + m_2^2}\|_{L^p}.
\end{equation}
Vese and Osher \cite{vese2003modeling} firstly proposed the following minimization problem to overcome the difficulty of
computing \eqref{g-pro}:
\begin{equation}\label{VO}
    \inf_{(u,m)} \int_{\Omega} |\nabla u| + \frac{\alpha}{2} \int_{\Omega}|f- u - \mathrm{div}(m)|^2 + \lambda [ \int_{\Omega} ( \sqrt{m_1^2 + m_2^2} )^p ]^{\frac{1}{p}},
\end{equation}
where $\alpha,\lambda > 0$ are tuning parameters, and the second term ensures that $f \approx u + \mathrm{div}(m)$.
If the parameter $\alpha \to \infty$, and $p \to \infty$, this model is formally an approximation of the Meyer's G-TV model \eqref{g-pro}. 
In numerical experiment, the authors \cite{vese2003modeling} considered the case $p=1$ and minimized \eqref{VO} using the associated Euler-Lagrange equation and gradient descent. 

The $G$ norm is replaced by the negative Sobolev norm in \eqref{VO} which seems to be useful to model oscillating patterns \cite{aujol2005dual}. 
The semi-norm in Sobolev space is
\begin{equation}
    \|u\|_{1,p} = \||\nabla u|\|_{L^p} = ( \int_{\Omega} |\nabla u|^p)^{1/p},\quad \forall p \geq 1.
\end{equation}
The negative Sobolev norm with $1/p + 1/q = 1$ is
\begin{equation}\label{neg}
    \begin{split}
        \|v\|_{-1,p} & = \sup_{\|u\|_{1,q}=1}(v,u) = \sup_{\||\nabla u|\|_{L^q}=1}(v,u) \\
        & = \inf \left\{ \||m|\|_{L^p} : v = \mathrm{div}(m),m \in L^{p}(\Omega)\ \right\}.    
    \end{split}
\end{equation}
The TV norm is just the semi-norm $\|u\|_{1,1}$, and the $G$ norm can be written as $\|v\|_G = \|v\|_{-1,\infty} =\sup_{\||\nabla u|\|_{L^1}=1}(v,u)$. As a matter of fact, the Wasserstein-1 distance \eqref{KR-1} is a special case of the negative Sobolev norm \eqref{neg} \cite{Peyr2019Computational}.

In \cite{osher2003image}, Osher, Sole, and Vese considered $p=2$, with $v = \mathrm{div}(m)$ and $m\in L^2(\Omega)$ in \eqref{VO}.
Then the norm $\|v\|_{-1,2} = \sqrt{\int_{\Omega} (m_1^2 + m_2^2) dx}$ is exactly the seminorm of $H^{-1}(\Omega)$, the dual of the space $H^{1}(\Omega)$ endowed with $\|g\|_{H^{1}(\Omega)} = \int_{\Omega} |\nabla g|^2 dx$.
Let $m = (m_1,m_2) = \nabla g$, then $f-u = \mathrm{div}(m) = \mathrm{div}(\nabla g) = \Delta g$, i.e., $g = \Delta^{-1}(f-u)$.
A convex minimization problem was proposed to approximate \eqref{g-pro}:
\begin{equation}\label{osv}
    \inf_u  \left\{ \int_{\Omega} |\nabla u| + \lambda \int_{\Omega} |\nabla (\Delta^{-1})(f-u)|^2 dx \right\},
\end{equation}
where $\|\cdot\|_{H^{-1}(\Omega)} = \int_{\Omega} |\nabla (\Delta^{-1}) (\cdot)|^2 dx$ is the norm in $H^{-1}(\Omega)$,
which was generalized to the negative Hilbert-Sobolev space $H^{-s}$ by Lieu and Vese in \cite{lieu2008image}.





A numerical approach for the Osher and Vese's functional \eqref{VO} was proposed in \cite{2005Image,2006Structure} by solving the following problem 
\begin{equation}\label{g-pro-l2}
    \inf_{(u,v)} J(u) + \frac{\alpha}{2} \|f-u-v\|_{L^2}^2 + \mathbf{I}_{G_{\mu}}(v), 
\end{equation}
where $J(u)=\mathrm{TV}(u)$, and $\mathbf{I}_{G_{\mu}}$ is the indicator function of the closed convex set $G_{\mu} = \{v \in G : \|v\|_G \leq \mu\}$.
The parameter $\mu$ plays the same role as $\lambda$ in problem \eqref{VO}. 
The convex functional \eqref{g-pro-l2} is minimized alternately in the two variables $u$ and $v$. Each minimization can be solved by Chambolle’s projection method \cite{2004cha}.

    Other related models for simultaneous image restoration and image decomposition into cartoon and texture have been proposed in the literature.

In \cite{aujol2005dual}, Aujol and Chambolle compared the properties of norms that are dual to negative Sobolev \eqref{neg} and Besov norms in the space $E = \dot{B}^{\infty}_{-1,\infty}$, for modeling texture and noise.
They proposed a model to decompose an image $f=u+v+w$ into $u \in BV$ the cartoon component, $v \in G$ the texture component, and $w \in E$ the noise component.
 
In \cite{2003StaImage}, Starck, Elad and Donoho proposed a decomposition model using wavelets dictionaries.
Image restoration models using wavelet frames to sparsely represent the cartoon part and the texture part separately have been studied in \cite{cai2010,2012Wavelet,cai2012image}. 

Generalizing the idea of Meyer \cite{meyeroscillating}, where $v \in F = \mathrm{div}(BMO)=\dot{BMO}^{-1}$, 
Garnett et al. \cite{garnett2011modeling} proposed to model the texture component $v$ by the action of the Riesz potentials on $v$ that belongs to $BMO$ or $L^p,1\leq p\leq \infty$, which was extended to an image deblurring model by Kim and Vese \cite{kim2009image}. 

In \cite{6459598}, Michael K. Ng et al. firstly proposed an image decomposition model which can simultaneously restore blurred images with missing pixels:
\begin{equation}\label{CT}
    \inf_{u \in \mathbb{R}^n,g \in \mathbb{R}^n \times \mathbb{R}^n} \||\nabla u|\| + \frac{\alpha}{2} \int_{\Omega} |f-K(u+\mathrm{div}(m))|^2 dx + \lambda \||m|\|_{L^p},
\end{equation}
where $p = 1,2$ and $\infty$, and different choices of $K$ correspond to the observed images with different corruptions. 
For solving \eqref{CT}, the alternating direction method of multipliers (ADMM) with a Gaussian back substitution procedure proposed in \cite{2012He} was recommended.
In \cite{doi:10.1137/140967416}, a color image decomposition model based on \eqref{CT} with higher-order regularizers was proposed.


\section{The Proposed Image Restoration Models} 
\label{sec:1}
    \subsection{The proposed image denoising model}

We propose a denoising model with the dual Lipschitz norm from optimal transport and the total variation:
\begin{equation}\label{kr_new_rel}
    \inf_{(u,v)\in BV\times H} \int_{\Omega} |\nabla u| + \frac{\alpha}{2} \int_{\Omega} |f-u-v|^2 dx + \lambda \|v\|_{\mathrm{Lip}^*},
\end{equation}
where $f:\Omega \to \mathbb{R} \in L^2(\mathbb{R}^2)$ is an observed noisy image and $\alpha,\lambda>0$ are tuning parameters. $v$ is defined on the space $H = \{v : v \in L^1(\Omega), \int_{\Omega} v(x) dx = 0\}$. 

The third term is the dual Lipschitz norm of $v=\mathrm{div}(m)$, which has an equivalent formulation considering the dual problem \eqref{KR-dual} 
of the $W_1$ distance \eqref{KR-1}, 
    \begin{equation}\label{lip-dual}
        \|v\|_{\mathrm{Lip}^*} = \inf_{m} \left\{ \||m|\|_{L^{1}} : \mathrm{div} (m) = v \right\},
    \end{equation}
which is only defined for $v \in H$ from the conservation of the mass constraint \cite{Peyr2019Computational}, and is closely related to the $G$-norm \eqref{gnorm}.

In this decomposition, $f$ is decomposed into $u+v+w$, with a piecewise smooth component $u \in BV$, component $v$ to model texture or noise, and a small residual $w$.
The parameter $\alpha$ controls the $L^2$ norm of the residual $w=f-u-v$ and ensures that $f-u \approx v$ as $\alpha \to \infty$.


In the KR-TV model \eqref{KR-pro}, the KR norm has a dual formulation \eqref{KR-pro-dual}, and $\lambda_1$ is so large that $f-u \approx \mathrm{div}(m)$ in the experiment \cite{2014Imaging}.

In our model \eqref{kr_new_rel}, we use the dual Lipschitz norm \eqref{lip-dual} of $v = \mathrm{div} (m)$ where the objective $|m|$ is penalized in the $L^{1}$ norm, while Osher and Vese used the negative Sobolev norms where $|m|$ is penalized in the $L^{p}(p\to \infty)$ norm in \eqref{VO} to approximate $G$-norm \eqref{gnorm}.



    \subsection{Existence and uniqueness of minimizers}

The existence and uniqueness of minimizers of the exact decomposition \eqref{Reg} using the Wasserstein metric was analyzed in \cite{2012Regularized}, and properties of minimizers of the KR-TV denoising model \eqref{KR-pro} was derived in \cite{2014Imaging}. 
 
Our model is closely related to the functional \eqref{VO} to approximate Meyer’s model.
Theoretical results for the existence and uniqueness of models related to Meyer's model have been investigated in \cite{2005Image,2005ImageBMO,osher2003image,garnett2011modeling,2007GarImage}.
In what follows, we show the existence and uniqueness of minimizers for the proposed model \eqref{kr_new_rel} as a generalization of the results in these works.  
For this, let's first quote a lemma from \cite{osher2003image}.
\begin{lemma}\label{H-1}
    Let $V_0 = \{g \in H^{1}(\Omega): \int_{\Omega} g(x) dx = 0\}$. If $v\in L^2(\Omega)$, with $\int_{\Omega} v(x) dx = 0$, then the problem
    \begin{equation}
        -\Delta g = v,\quad \frac{\partial g}{\partial n}|_{\Omega} = 0,
    \end{equation}
    admits a unique solution $g$ in $V_0$.
\end{lemma}

We then prove the following theorem. Let $|u|_{BV(\Omega)} = \int_\Omega |\nabla u|$, and $\|u\|_{BV} = \|u\|_{L^1} + \int_{\Omega} |\nabla u|$.

\begin{theorem}\label{thm-1}
    Let $f \in L^2(\Omega)$, then there exists $(u,v) \in (BV, H)$ where $H = \{v : v \in L^1(\Omega), \int_{\Omega} v(x) dx = 0\}$, such that $(u,v)$ is a minimizer of the energy 
    \begin{equation}\label{kr_new_rel_fun}
         \inf_{(u,v)} \left\{ F(u,v) = \int_{\Omega}|\nabla u| + \frac{\alpha}{2} \int_{\Omega}|f-u-v|^2 + \lambda \|v\|_{\mathrm{Lip}^*} \mathbf{I}_{v \in H},\int_{\Omega} f = \int_{\Omega} u  \right\}.
    \end{equation}
    If, in addition, $\int_{\Omega} f \neq 0$, then the minimizer is unique.
\end{theorem}

\begin{proof}
\begin{itemize}
    \item Existence.
    Assume $\Omega=[0,1]^2$. Let $(u_k,v_k)$ be a minimizing sequence of \eqref{kr_new_rel_fun}. 
    For each $u_k$, $u_k+c_k$ is another minimizing sequence such that $c_k = \mathrm{arg}\min_{c} \int_{\Omega}|f-(u_k+c)-v_k|^2 dx$. The unique minimizer of this
    quadratic function in $c_k$ is $c_k=\int_{\Omega}(f-u_k-v_k)dx$.
    Thus, the minimizing sequence $u_k+c_k$ such that $c_k-\int_{\Omega}(f-u_k-v_k)dx=\int_{\Omega}(f-(u_k+c_k)-v_k)dx = 0$. Therefore, we can assume the minimizing sequence $(u_k,v_k)$ satisfies $\int_{\Omega}(f-u_k-v_k)dx=0$.
    We know that, for $v_k =\mathrm{div}(m_k) \in H$, we have $\int_{\Omega} v_k(x) dx = 0$. Therefore, we deduce $\int_{\Omega} f dx = \int_{\Omega} u_k dx$.
    There is $0<C<\infty$ such that the following uniform bounds hold, for all $k$,
    \begin{eqnarray}
        & \int_{\Omega}|\nabla u| \leq C,\label{ieq1}\\
        & \frac{\alpha}{2} \int_{\Omega}|f-u_k-v_k|^2 \leq C,\label{ieq2}\\
        & \|v\|_{\mathrm{Lip}^*}\mathbf{I}_{v \in H} \leq C.\label{ieq3}
    \end{eqnarray}
    By the Poincar\'{e} inequality,
    \begin{equation}
        \|u_k - \frac{1}{|\Omega|}\int_{\Omega} u_k\|_{L^2} \leq C \int_{\Omega} |\nabla u|,
    \end{equation}
    where $\frac{1}{|\Omega|}\int_{\Omega} u_k = \frac{1}{|\Omega|}\int_{\Omega} f$, therefore we have $u_k$ is uniformly bounded in $L^2$, and $\|u_k\|_{L^1} \leq C' \|u_k\|_{L^2} \leq C,\quad \forall k$. This implies that 
    \begin{equation}
        \|u_k\|_{BV} = \|u_k\|_{L^1} + \int_{\Omega} |\nabla u_k| \leq C,\quad \forall k,
    \end{equation}
    which means that there exists a function $u \in BV(\Omega)$ such that there exists a subsequence, denoted $u_{n_k}$, converges to $u$ strongly in $L^1$, and we have
    \begin{equation}
        \int_{\Omega} |\nabla u| \leq \liminf_{k\to \infty} \int_{\Omega} |\nabla u_k|.
    \end{equation}

    The condition \eqref{ieq2} and the fact that $f-u_k \in L^2(\Omega)$ is uniformly bounded, imply that $v_k \in L^2(\Omega)$, with $\int_{\Omega} v(x) dx = 0$.
    From Lemma \ref{H-1}, we can associate a unique $g_k \in H^1(\Omega)$ such that $v_k = -\Delta g_k$ and also $\int_{\Omega} g_k dx = 0$.

    The condition \eqref{ieq3} implies that $\|v_k\|_{\mathrm{Lip}^*} = \||\nabla g_k|\|_{L^1}$ is uniformly bounded.
    By the Poincar\'{e} inequality,
    \begin{equation}
        \|g_k - \frac{1}{|\Omega|}\int_{\Omega} g_k\|_{L^2} = \|g_k\|_{L^2} \leq C \int_{\Omega} |\nabla g_k|,\quad \forall k,
    \end{equation}
    we have $\|g_k\|_{L^1} \leq C' \|g_k\|_{L^2} \leq C,\quad \forall k$ similarly and
    \begin{equation}
        \|g_k\|_{BV} = \|g_k\|_{L^1} + \int_{\Omega} |\nabla g_k| \leq C,\quad \forall k,
    \end{equation}
    which means that there exists $g \in BV(\Omega)$ such that, there exists a subsequence, denoted $g_{n_k}$, converges to $g$ strongly in $L^1$, and
    \begin{equation}
        \int_{\Omega} |\nabla g| \leq \liminf_{k\to \infty} \int_{\Omega} |\nabla g_k|.
    \end{equation}
    We also have $v = - \Delta g$ and $\|v\|_{\mathrm{Lip}^*} = \||\nabla g|\|_{L^1}$, therefore, we obtain that $\|v\|_{\mathrm{Lip}^*} \leq \liminf_{k\to \infty} \|v_k\|_{\mathrm{Lip}^*}$.
    
    Similarly, we have $\int_{\Omega}|f-u-v|^2 \leq \liminf_{k\to \infty} \int_{\Omega}|f-u_k-v_k|^2$.

    Finally, we deduce that $F(u,v) \leq \liminf_{k\to \infty} F(u_k,v_k)$ and therefore $(u,v)$ minimizes $F$.

    \item Uniqueness. 
    If $(u,v)$ are minimizers of \eqref{kr_new_rel_fun}, we can show that $\int_{\Omega} f dx = \int_{\Omega} u dx$ as in \cite{2007GarImage}.
    Since $\int_{\Omega} |\nabla u| = \int_{\Omega} |\nabla (u+c)|$ for every constant $c$, we also have
    \begin{equation}
        \min_{c \in \mathbb{R}} \int_{\Omega}|f-(u+c)-v|^2 dx = \int_{\Omega}|f-u-v|^2 dx.
    \end{equation}
    But $c=0$ is the unique minimizer of this quadratic function, and in addition $c = \int_{\Omega}(f-u-v)dx$. Thus, we have $\int_{\Omega}(f-u-v)dx = 0$. 
    Therefore $\int_{\Omega} f dx = \int_{\Omega} u dx$, since $\int_{\Omega}v(x)dx=0$ using $v\in H$. 

    As in \cite{2005Image,2005ImageBMO}, the functional $F(u,v)$ in \eqref{kr_new_rel_fun} is a sum of two convex functions, and of a strictly convex function except in the direction $(u,-u)$.
    Therefore, it suffices to check that if $(\hat{u},\hat{v})$ is a minimizer, then $(\hat{u}+t\hat{u},\hat{v}-t\hat{u})$ is not a minimizer for $t \neq 0$. Since $(\hat{u},\hat{v})$ is a minimizer, then $\int_{\Omega} \hat{u} = \int_{\Omega} f$. Therefore, if $(\hat{u}+t\hat{u},\hat{v}-t\hat{u})$ is a minimizer too, then $\int_{\Omega} (1+t)\hat{u} = \int_{\Omega} f$, which implies $t=0$. Finally, we conclude the uniqueness.
    
\end{itemize}
$\hfill\square$
\end{proof}

\subsection{Characterization of minimizers}
The effect of the ROF minimization was analyzed in \cite{meyeroscillating}, which showed that the ROF model cannot split an image $f$ into a sum $u+v$ where $u$ would model the object and $v$ would be the texture or noise components.

Let $D$ be a disc centered at 0 with radius $R$, Meyer showed that by applying the ROF model \eqref{rof} to $f=a\mathbf{I}_{D}$ where $a>0$ is a constant and assume $\alpha R \geq 1/a$, the decomposition is given by
\begin{equation}\label{eqn:rof}
    f=u+v,\quad u = (a - (\alpha R)^{-1})\mathbf{I}_{D},\quad v = (\alpha R)^{-1}\mathbf{I}_{D}.
\end{equation}
Notice that $v$ is independent of $a$ and is not the texture and noise component. $u$ cannot be $f$ for any finite $\alpha$.
For proving $\|\mathbf{I}_{D}\|_{G} = R/2$, it suffices to use the following lemma from \cite{meyeroscillating}.
\begin{lemma}\label{lem:rof}
    Let $f,u,v \in L^2(\mathbb{R}^2)$, then for the norm $\|\cdot\|_{G}$ \eqref{gnorm}, if $\|f\|_{G} > \frac{1}{\alpha}$, the decomposition $f=u+v$ by using the ROF model \eqref{rof} is characterized by the following two conditions
    \begin{equation}
        \|v\|_{G} = \frac{1}{\alpha},\quad \int u(x)v(x)dx = \frac{1}{\alpha}|u|_{BV}.
    \end{equation}
\end{lemma}

We show some properties of the minimizers of model \eqref{kr_new_rel} as a generalization of Lemma \ref{lem:rof}.
Theoretical results including the characterization of minimizers of models related to Meyer’s model were presented in \cite{meyeroscillating,osher2003image,2005ImageBMO,2007GarImage,garnett2011modeling}.

\begin{theorem}\label{thm:lip-w}
    Given a function $w \in L^2(\Omega)$ and $\lambda>0$, define
    \begin{equation}\label{eq:norm}
    \|w\|_{*,\lambda} = \sup_{g\in BV,h\in H} \frac{\langle w, g+h \rangle}{|g|_{BV} + \lambda \|h\|_{\mathrm{Lip}^*}},\quad |g|_{BV} + \lambda \|h\|_{\mathrm{Lip}^*} \neq 0,
    \end{equation}
    where $\langle \cdot,\cdot \rangle$ is the $L^2$ inner product.
    Let $(u,v)$ be an optimal decomposition of $f \in L^2(\Omega)$ via \eqref{kr_new_rel_fun}, and denote $w=f-u-v$. Then we have the following:
    \begin{enumerate}
        \item $\|f\|_{*,\lambda} \leq \frac{1}{\alpha} \Longleftrightarrow u=0,v=0$ and $w=f$.
        
        \item Assume $\|f\|_{*,\lambda} > \frac{1}{\alpha}$, then $(u,v)$ is characterized by the two conditions
        \begin{equation}
            \|w\|_{*,\lambda} = \frac{1}{\alpha}, \quad
        \langle w, u+v \rangle = \frac{1}{\alpha}(|u|_{BV} + \lambda \|v\|_{\mathrm{Lip}^*}).
        \end{equation}
    
    \end{enumerate}
\end{theorem}

\begin{proof}
    \begin{enumerate}
        \item $(u,v)=(0,0)$ is a minimizer of the functional \eqref{kr_new_rel_fun} if and only if for any $g\in BV,h\in H$ and $\epsilon \in \mathbb{R}$,
        \begin{equation}
            |\epsilon g|_{BV(\Omega)} + \frac{\alpha}{2} \int_{\Omega}|f-\epsilon (g+h)|^2 + \lambda \|\epsilon h\|_{\mathrm{Lip}^*} \geq \frac{\alpha}{2} \int_{\Omega}|f|^2.
        \end{equation}
        By taking $\epsilon \to 0$, the above equation holds if and only if $\langle f,g+h \rangle \leq \frac{1}{\alpha} (|g|_{BV(\Omega)} + \lambda \|h\|_{\mathrm{Lip}^*})$.
        By the definition of \eqref{eq:norm}, we have $\|f\|_{*,\lambda} \leq \frac{1}{\alpha}$.
        
        For the converse implication,
        assume for any $g\in BV,h\in H$, 
        \begin{equation}
            \langle f,g+h \rangle \leq (|g|_{BV} + \lambda \|h\|_{\mathrm{Lip}^*})\|f\|_{*,\lambda} \leq \frac{1}{\alpha} (|g|_{BV} + \lambda \|h\|_{\mathrm{Lip}^*}).
        \end{equation}
        We have \begin{equation}
            \begin{split}
            & |g|_{BV(\Omega)} + \frac{\alpha}{2} \int_{\Omega}|f-(g+h)|^2 + \lambda \| h\|_{\mathrm{Lip}^*} = \|g\|_{BV(\Omega)} + \frac{\alpha}{2}( \int_{\Omega}|f|^2 + \int_{\Omega}|g+h|^2 - 2\langle f,g+h \rangle ) + \lambda \| h\|_{\mathrm{Lip}^*}\\
            & \geq |g|_{BV(\Omega)} + \frac{\alpha}{2}( \int_{\Omega}|f|^2 + \int_{\Omega}|g+h|^2) + \lambda \| h\|_{\mathrm{Lip}^*} - (|g|_{BV} + \lambda \|h\|_{\mathrm{Lip}^*})
            \geq \frac{\alpha}{2} \int_{\Omega}|f|^2.
            \end{split}
        \end{equation}
        Thus $u=0$ and $v=0$ give the optimal decomposition in this case.
        
        \item Suppose $\|f\|_{*,\lambda} > \frac{1}{2\alpha}$. If $(u,v)$ is a minimizer, then we have $u \neq 0$ or $v \neq 0$.
        For any $g\in BV,h\in H$ and $\epsilon \in \mathbb{R}$,
        \begin{equation}\label{eq:epi}
            \begin{split}
                &|u+\epsilon g|_{BV(\Omega)} + \frac{\alpha}{2} \int_{\Omega}|w-\epsilon (g+h)|^2 + \lambda \|v+\epsilon h\|_{\mathrm{Lip}^*} \geq |u|_{BV(\Omega)} + \frac{\alpha}{2} \int_{\Omega}|w|^2 + \lambda \|v\|_{\mathrm{Lip}^*}\\
                & \Longrightarrow |u|_{BV(\Omega)} +|\epsilon||g|_{BV(\Omega)} + \frac{\alpha}{2} \int_{\Omega}|w-\epsilon (g+h)|^2 + \lambda (\|v\|_{\mathrm{Lip}^*}+|\epsilon|\|h\|_{\mathrm{Lip}^*}) \geq |u|_{BV(\Omega)} + \frac{\alpha}{2} \int_{\Omega}|w|^2 + \lambda \|v\|_{\mathrm{Lip}^*}\\
                & \Longrightarrow |\epsilon||g|_{BV(\Omega)} + \frac{\alpha}{2} ( \int_{\Omega}|w|^2 + \epsilon^2 \int_{\Omega} |g+h|^2 - 2\epsilon \langle w,g+h \rangle) + |\epsilon| \lambda \|h\|_{\mathrm{Lip}^*} \geq \frac{\alpha}{2} \int_{\Omega}|w|^2.
            \end{split}
        \end{equation}
        Dividing both side of the last equation by $\epsilon>0$, we obtain
        \begin{equation}
            \alpha \langle w,g+h \rangle \leq |g|_{BV(\Omega)} + \lambda \|h\|_{\mathrm{Lip}^*} + \frac{\alpha \epsilon}{2} \int_{\Omega} |g+h|^2.
        \end{equation}
        Taking $\epsilon \to 0$, we obtain 
        \begin{equation}\label{eqn:w}
            \|w\|_{*,\lambda} \leq \frac{1}{\alpha}.
        \end{equation}

        If we take $(g,h) = (u,v)$ in the first equation in \eqref{eq:epi}, and take $-1 < \epsilon < 1$, then \eqref{eq:epi} implies
        \begin{equation}
            \begin{split}
            & (1+\epsilon) (|u|_{BV(\Omega)} + \lambda\|v\|_{\mathrm{Lip}^*}) + \frac{\alpha}{2} \int_{\Omega}|w-\epsilon (u+v)|^2 \geq |u|_{BV(\Omega)} + \frac{\alpha}{2} \int_{\Omega}|w|^2 + \lambda \|v\|_{\mathrm{Lip}^*}\\
            & \Longrightarrow \epsilon \alpha \langle w,u+v \rangle \leq \epsilon (|u|_{BV(\Omega)} + \lambda\|v\|_{\mathrm{Lip}^*}) + \frac{\epsilon^2 \alpha}{2} \int_{\Omega} |u+v|^2.
            \end{split}
        \end{equation}
        If $\epsilon < 0$ and letting $\epsilon \to 0$, we obtain
        \begin{equation}\label{eqn:inv}
            \alpha \langle w,u+v \rangle \geq (|u|_{BV(\Omega)} + \lambda \|v\|_{\mathrm{Lip}^*}).
        \end{equation}
        Combining \eqref{eqn:inv} and \eqref{eqn:w}, we obtain the desired results,
        \begin{equation}\label{eqn:min}
            \|w\|_{*,\lambda} = \frac{1}{\alpha},\quad \alpha \langle w,u+v \rangle = (|u|_{BV(\Omega)} + \lambda \|v\|_{\mathrm{Lip}^*}).
        \end{equation}

        Conversely, if \eqref{eqn:min} holds for some $(u,v,w)$, then for any $g\in BV,h\in H$ and $\epsilon \in \mathbb{R}$, we have
        \begin{equation}
            \begin{split}
                & |u+\epsilon g|_{BV(\Omega)} + \lambda \|v+\epsilon h\|_{\mathrm{Lip}^*} + \frac{\alpha}{2} \int_{\Omega}|w-\epsilon (g+h)|^2  \\
                & \geq \alpha \langle w,u+\epsilon g + v+\epsilon h \rangle + \frac{\alpha}{2}( \int_{\Omega}|w|^2 + \epsilon^2 \int_{\Omega} |g+h|^2 - 2\epsilon \langle w,g+h \rangle )\\
                & = \frac{\alpha}{2}\langle w,u+v \rangle + \frac{\alpha}{2} \int_{\Omega}|w|^2 = |u|_{BV(\Omega)} + \lambda \|v\|_{\mathrm{Lip}^*} + \frac{\alpha}{2} \int_{\Omega}|w|^2.
            \end{split}
        \end{equation}
        Therefore, $(u, v)$ is an optimal decomposition of \eqref{kr_new_rel_fun}.
    \end{enumerate}
    $\hfill\square$
\end{proof}

\subsection{An extension of our model}

In order to suppress the staircasing effect and preserve the image contrast, we propose to apply a variant of total variation proposed in \cite{zhu} as regularization in our model \eqref{kr_new_rel},
\begin{equation}\label{kr_new_rel_log}
    \inf_{(u,v)\in BV\times H} \int_{\Omega} R(\nabla u) + \frac{\alpha}{2} \int_{\Omega} |f-K * u-v|^2 dx + \lambda \|v\|_{\mathrm{Lip}^*},
\end{equation}
where $K$ is a convolution operator and $K * u$ means the convolution of $K$ and $u$. In the purely denoising model, $K=I$ where $I$ is an identity matrix. 
The potential function $R(\cdot)=\phi_a(\cdot)$ reads \begin{equation}
    \begin{split}
    \phi_a(x) = \left \{
        \begin{array}{lr}
            \frac{1}{2a} x^2,                 & |x|\leq a,\\
            a \ln |x| + \frac{a}{2} - a\ln a, & |x|> a.
        \end{array}
        \right.	
    \end{split}
\end{equation}
In fact, when $R(x)=|x|=\sqrt{x_1^2+x_2^2}$, we obtain the total variation $R(\nabla u)=|u|_{BV(\Omega)}$, and different potential functions have been proposed in the literature \cite{2018On,2001A}. 

The first two terms in \eqref{kr_new_rel_log} are the same as the original image restoration model proposed in \cite{zhu}, which reads
\begin{equation}\label{log-tv}
     \inf_{u\in BV} \int_{\Omega} \phi_a(\nabla u) + \frac{\alpha}{2} \int_{\Omega} |f-K * u|^2 dx.
\end{equation}


\section{Numerical Algorithm}\label{sec:alg}
    

\subsection{Discretization}\label{Discretization}


Suppose $\Omega = [0,L]^2$, where $L$ is the length of the interval. Let $h$ be the mesh size, $\Omega^h = \{(ih,jh):1 \leq i \leq N,\quad 1 \leq j \leq N \}$ is the discretized image domain. 
In \cite{2021MULTILEVEL}, $\Omega=[0,1]^2$ for the calculation of the Wasserstein-1 distance.
The distributions $u,f$ and $v$ are discretized as $N \times N$ tensors, the discrete flux $m = (m_1, m_2)$ where $m_1$ and $m_2$ are $N\times N$ tensors, which represents a map $u = \{u(x)\}_{x\in \Omega^h},f = \{f(x)\}_{x\in \Omega^h}$, and $m = \{m(x)\}_{x\in \Omega^h}$.
For simplicity, we use the same symbol for the continuous variables and their discretizations.

Define some norms
\begin{equation}\label{norm}
    \begin{aligned}
        \|m\|_2^2 = \sum_{i,j} |m_{(i,j)}|^2,\quad \|\phi\|^2 = \sum_{i,j} \phi_{(i,j)}^2,\quad \|m\|_{1,2} = \sum_{i,j} |m_{(i,j)}|,
    \end{aligned}  
\end{equation}
where $|\cdot|$ is the standard Euclidean norm.

Define inner products
\begin{equation}
    \langle \phi,\phi'\rangle = \sum_{i,j}\phi_{(i,j)}\phi_{(i,j)}',\quad
    \langle m,n\rangle = \langle m_1,n_1\rangle + \langle m_2,n_2\rangle.
\end{equation}

Define the discrete divergence operator $\mathrm{div} (m) = \partial_x^- m_1 + \partial_y^- m_2$ with homogeneous Dirichlet boundary conditions as
\begin{align}\label{div}
    \begin{aligned}
        \partial_x^- m_1 = \left \{	
        \begin{array}{lr}
            \frac{- m_{1(N-1,j)}}{h}, & i = N,\\[1mm]
            \frac{m_{1(i,j)} - m_{1(i-1,j)}}{h},  & 1 < i< N,\\[1mm]
            \frac{m_{1(1,j)}}{h}, & i=1.\\
        \end{array}
    \right.	
    \end{aligned}	
    \quad
    \begin{aligned}
        \partial_y^- m_2 = \left \{	
        \begin{array}{lr}
             \frac{- m_{2(i,N-1)}}{h}, & j = N,\\[1mm]
            \frac{m_{2(i,j)} - m_{2(i,j-1)}}{h},  &1 < j < N,\\[1mm]
            \frac{m_{2(i,1)}}{h}, & j=1.\\
        \end{array}
    \right.	
    \end{aligned}	
\end{align}
The definition of the discrete $\mathrm{div} (m)$ consistent with the zero-flux boundary conditions, i.e. $m(x) \cdot n(x) = 0, \forall x \in \partial \Omega$, where $n(x)$ denotes the normal vector at $x$ \cite{2017Vector}.



As in \cite{2010Augmented}, denote $V = \mathbb{R}^{N\times N}$ and $Q = V \times V$, then the discrete gradient operator with periodic boundary condition $\nabla  = (\partial_x^+, \partial_y^+): V \to Q$ is defined as 
\begin{align*}
    \begin{aligned}
       \partial_x^+ u = \left \{	
        \begin{array}{lr}
            \frac{u_{(i+1,j)} - u_{(i,j)}}{h},  & 1 \leq i\leq N-1,\\
            \frac{u_{(1,j)} - u_{(N,j)}}{h}, & i=N.\\
        \end{array}
    \right.	
    \end{aligned}	
    \quad
    \begin{aligned}
       \partial_y^+ u = \left \{	
        \begin{array}{lr}
            \frac{u_{(i,j+1)} - u_{(i,j)}}{h},  & 1 \leq i\leq N-1,\\
            \frac{u_{(i,1)} - u_{(i,N)}}{h}, & j=N.\\
        \end{array}
    \right.	
    \end{aligned}	
\end{align*}

The discretized minimization problem associated to \eqref{kr_new_rel} and \eqref{kr_new_rel_log} is given by:
\begin{equation}\label{kr_dis}
  \inf_{(u,v)\in V \times V} F_{\alpha,\lambda}(u,v) = \|R(\nabla u)\|_{1,2} + \frac{\alpha}{2} \|f-K*u-v\|^2 + \|v\|_{\mathrm{Lip}^*}.
\end{equation}

\subsection{Numerical Method for problem \eqref{kr_dis}}\label{sec:alg-kr}

Motivated by the numerical method for $G$-TV model \eqref{g-pro-l2} in \cite{2005Image,2006Structure}, we propose to minimize the proposed models alternately and repeatedly 
in the two variables $u$ and $v$.
Specifically, we consider the following two problems:
\begin{itemize}
    \item $u$ being fixed, we search for $v$ as a solution of:
    \begin{equation}\label{pro1}
        \inf_{v} \frac{\alpha}{2} \|f-K*u-v\|^2 + \|v\|_{\mathrm{Lip}^*}.
    \end{equation}
    
    \item $v$ being fixed, we search for $u$ as a solution of:
    \begin{equation}\label{pro2}
        \inf_{u}  \|R(\nabla u)\|_{1,2} + \frac{\alpha}{2} \|f-v - K*u\|^2.
    \end{equation}
\end{itemize}

\subsubsection{Numerical Method for \eqref{pro1}}\label{sec:pdhg}
The discretized version of Wasserstein-1 distance \eqref{KR-1} and \eqref{KR-dual} can be jointly solved by the following Lagrangian functional \cite{2021MULTILEVEL}, 
\begin{equation}\label{W1PD}
    \inf_{m} \sup_{\phi} \left\{ L(m, \phi) = \|m\|_{1,2} + \langle \phi, \mathrm{div} (m) - \mu + \nu \rangle \right\}.
\end{equation}
A primal-dual algorithm was proposed in \cite{2016A} for \eqref{W1PD}.

The dual Lipschitz norm \eqref{lip-dual} is defined by the Wasserstein-1 distance \eqref{KR-dual} as $\|v\|_{\mathrm{Lip}^*}=W_1(\mu,\nu)$ where $v=\mu-\nu$.
For fixed $f_1 = f - K*u$, the subproblem \eqref{pro1} can be written as
    \begin{equation}\label{pro1-ot}
        \begin{split}
        & \min_{v, m}\,\, \frac{\alpha}{2} \|f_1-v\|^2 + \lambda \|m\|_{1,2}, \\
        & s.t. \quad \mathrm{div} (m) = v,\; \forall x \in \Omega,
        \end{split}
    \end{equation}    
By the definition of $\mathrm{div}$ \eqref{div}, we have $\sum_{i,j} v_{(i,j)} = \sum_{i,j}\mathrm{div}(m)_{(i,j)}=0$.

The solution of \eqref{pro1-ot} is the saddle point of the following Lagrangian function, 
    \begin{equation} \label{vpro}
        \min_{v}\min_{m} \max_{\phi} \left\{ L(v,m,\phi) = \langle \mathrm{div} (m) - v ,\phi \rangle
             + \lambda\|m\|_{1,2} + \frac{\alpha}{2} \|f_1-v\|^2 \right\}.
    \end{equation} 


To solve \eqref{vpro}, we use the Primal-Dual Hybrid Gradient (PDHG) method proposed for the following convex-concave saddle function \cite{2017Vector},
    \begin{equation}
        L(x,y,z) = f(x) + g(y) + \langle Ax+By, z\rangle,
    \end{equation}
where $f,g,h$ are (closed and proper) convex functions and $A,B,C$ are linear matrix. Assume $L$ has a saddle point, then the method
\begin{equation}\label{pdhg}
    \begin{split}
        & x^{k+1} = \mathrm{argmin}_x\{L(x,y^k,z^k) + \frac{1}{2\mu}\|x - x^k\|^2\},\\
        & y^{k+1} = \mathrm{argmin}_y\{L(x^k,y,z^k) + \frac{1}{2\nu}\|y - y^k\|^2\},\\
        & z^{k+1} = \mathrm{argmax}_w\{L(2x^{k+1}-x^k,2y^{k+1}-y^k,z) - \frac{1}{2\tau}\|z - z^k\|^2\},
    \end{split}
\end{equation}
converges 
when the step sizes $\mu, \nu, \tau>0$ satisfy
    \begin{equation}\label{stepsize}
        1 > \tau\mu \lambda_{\max}(A^TA) + \tau \nu \lambda_{\max}(B^TB).
    \end{equation}
PDHG method can be interpreted as a proximal point method \cite{2012Convergence}, and the convergence analysis can be simplified.


The iterative steps of PDHG algorithm for \eqref{vpro} is 
    \begin{equation}\label{pdhg-kr}
        \begin{split}
        & m^{k+1} = \mathrm{shrink}_m (m^k + \mu \nabla \phi^k; \mu \lambda), \\
        & v^{k+1} = (v^k/\nu + \alpha f_1 + \phi) / (\alpha + 1/\nu), \\
        & \hat{m}^{k+1} = 2 m^{k+1} - m^k, \hat{v}^{k+1} = 2 v^{k+1} - v^k, \\
        & \phi^{k+1} = \phi^k + \tau (\mathrm{div} (\hat{m}^{k+1}) - \hat{v}^{k+1}).     
    \end{split}
    \end{equation}

    The operater $\mathrm{shrink}_m$ in \eqref{pdhg-kr} for the norm $\|m\|_{1,2}$ has closed form solutions 
    \begin{equation}
        \mathrm{shrink}_m(x;\mu)_{(i,j)} = \mathrm{argmin}_{z}\{\mu \|z\| + 1/2 \|z-x_{(i,j)}\|^2\} = \max\{0,1-\frac{\mu}{\|x_{(i,j)}\|_2}\} x_{(i,j)}, \quad \forall 1\leq i,j\leq N.
    \end{equation}
    
To satisfy the convergence condition \eqref{stepsize} where $A^TA = \Delta, \lambda_{max}(-\Delta) \leq 8/h^2$, we can take $\mu = h^2/(32\tau), \nu=1/(4\tau)$, and $\tau$ is tuned for faster convergence of the algorithm.

We use the fixed point residual $R^k$ as a measure of progress and as a termination criterion,
    \begin{equation}\label{residual}
    \begin{split}
        \frac{1}{h^2} R^k = \frac{1}{\mu}\|m^{k+1} - m^{k}\|^2 + \frac{1}{\nu}\|v^{k+1} - v^{k} \|^2 + \frac{1}{\tau}\|\phi^{k+1} - \phi^{k} \|^2 \\
        - 2\langle\phi^{k+1} - \phi^{k}, \mathrm{div}(m^{k+1} - m^{k}) - (v^{k+1} - v^{k}) \rangle.
    \end{split}
    \end{equation}
Iteration \eqref{pdhg} stops until $R^k$ is below $\epsilon$.

    \subsubsection{Numerical Method for \eqref{pro2}}\label{fix}

We propose using the augmented Lagrangian method (ALM) in \cite{2010Augmented,2017Augmented,zhu,2023A} to solve subproblem \eqref{pro2}.
Aujol et al. \cite{2005Image} used a projection algorithm to minimize the total variation proposed by Chambolle \cite{2004cha}. 

An equivalent constraint problem of \eqref{pro2} for fixed $f_2 = f-v$ is as follows:
\begin{align*}
    & \|R(p)\|_{1,2} + \frac{\alpha}{2} \|K*u-f_2\|^2, \\
    & s.t. \quad p = \nabla u.
\end{align*}
The augmented Lagrangian functional is
\begin{equation}\label{lag}
    \begin{split}
        \mathcal{L}(u,p,\lambda_1) = \|R(p)\|_{1,2} + \frac{\alpha}{2} \|K*u-f_2\|^2 + \langle p-\nabla u,\eta \rangle  + \frac{r}{2} \|p-\nabla u\|^2,
    \end{split}
\end{equation}
with the Lagrange multipliers $\eta=(\eta_1,\eta_2) \in Q$ and $r > 0$.


To find the saddle point of the augmented Lagrangian, 
an alternative minimization is adopted: for each $u,p$, we fix the other variables and find the minimizer of the related subproblems. 
The subproblems in the $\ell$-th iteration are as follows:
\begin{equation}\label{subproblem}
\begin{aligned}
    & u^{\ell+1} = \mathrm{arg} \min_{u\in V} \mathcal{G}(u) = \frac{\alpha}{2} \|K*u-f_2\|^2 - \langle \nabla u,\eta^{\ell} \rangle  + \frac{r}{2} \|p^{\ell}-\nabla u\|^2,\\
    & p^{\ell+1} = \mathrm{arg} \min_{p \in Q} \mathcal{G}(p) = \|R(p)\|_{1,2} + \langle p, \eta^{\ell}\rangle  + \frac{r}{2} \|p - \nabla u^{\ell+1}\|^2.
\end{aligned}
\end{equation}

As discussed in \cite{2007A,2010Augmented,2017Augmented,zhu}, the solutions of $\min_{u\in V} \mathcal{G}(u)$ is determined by the optimality condition 
\begin{eqnarray}\label{subu}
    \mathrm{div} \eta^{\ell} + r \mathrm{div}(p^{\ell}-\nabla u) + \alpha K^*(K*u - f_2) = 0,
\end{eqnarray}
which can be efficiently solved via Fast Fourier Transform (FFT) implementation as follows,
\begin{equation}\label{fft}
    \begin{split}
    &  \mathcal{F}(\partial_1^-) (\mathcal{F}((\eta_1)^{\ell})+ r \mathcal{F}((p^1)^{\ell})) + \mathcal{F}(\partial_2^-) (\mathcal{F} -((\eta_2)^{\ell})+ r \mathcal{F}((p^2)^{\ell})) - \alpha \mathcal{F}(K^*) \mathcal{F}(f_2)\\
    & = (r \mathcal{F}(\Delta) - \alpha \mathcal{F}(K^*)\mathcal{F}(K))\mathcal{F}(u),
    \end{split}
\end{equation}
where $\partial_1^-,\partial_2^-,\Delta,K$ can be viewed as convolution operators by considering the periodic boundary conditions, and $\mathcal{F}(\cdot)$ is the two-dimensional Fourier transform function.

The subproblem $\min_{p \in Q} \mathcal{G}(p)$ has closed form solution for $R(\cdot) = |\cdot|$, let $w^{\ell+1} = \nabla u^{\ell+1} - \eta^{\ell} / r$,
\begin{equation}\label{subp}
    p_{(i,j)}^{\ell+1} = w_{(i,j)}^{\ell+1}\max\left\{0, 1-\frac{1}{r |w_{(i,j)}^{\ell+1}|}\right\}.
\end{equation}

The algorithm for the subproblem \eqref{pro2} is summarized in Algorithm \ref{alg2}, and the convergence results are given in \cite{2010Augmented,2017Augmented,zhu,2012He}.
\begin{algorithm}
    \caption{Augmented Lagrangian method (ALM) for the subproblem \eqref{pro2}.}
    \begin{algorithmic}
    \label{alg2}
    \REQUIRE initialize: $u^0,p^0$ and $\eta^0$.        
    \FOR {$\ell = 0$ to L}          

        \STATE Compute the minimizer $u^{\ell+1}$ by \eqref{fft} for the associated subproblems with the fixed $\eta^{\ell},p^{\ell}$.
         
        \STATE Compute the minimizer $p^{\ell+1}$ by \eqref{subp} for the associated subproblems with the fixed  $\eta^{\ell},u^{\ell+1}$.

        \STATE Update the Lagrangian multiplier $\eta$
        \begin{equation}
	       \eta^{\ell+1} = \eta^{\ell} + r (p^{\ell+1} - \nabla u^{\ell+1}).
        \end{equation}
        
        \STATE Measure the relative residuals and relative errors and stop the iteration if they are smaller than a given threshold.
    \ENDFOR
    \end{algorithmic}
\end{algorithm}



\subsection{Convergence of the Proposed Algorithm}

The algorithm to solve the proposed models is summarized in Algorithm \ref{alg_uv}, and the convergence of Algorithm \ref{alg_uv} can be proved by a straightforward generalization of the results in \cite{2005Image,2006Structure}.

\begin{algorithm}
    \caption{Numerical Method for Problem \eqref{kr_dis}.}
    \begin{algorithmic}
    \label{alg_uv}
    \REQUIRE Initialize: $u^0$ and $p^0$.        
    \FOR {$n = 0$ to N}          

        \STATE Compute the minimizer of the subproblem
        \begin{equation}
        v_{n+1} = \argmin_{v} \frac{\alpha}{2}\|f - K* u_n - v\|^2 + \lambda \|v\|_{\mathrm{Lip}^*}, 
        \end{equation} 
        by PDHG method \eqref{pdhg-kr} for its Lagrangian functional \eqref{vpro}.
         
        \STATE Compute the minimizer of the subproblem
        \begin{equation}
        u_{n+1} = \argmin_{u} \|R(\nabla u)\|_{1,2} + \frac{\alpha}{2} \|f - v_{n+1} - K * u \|^2
        \end{equation}
        by Algorithm \ref{alg2} for its Lagrangian functional \eqref{lag}.

    
    \ENDFOR
    \end{algorithmic}
\end{algorithm}

\begin{proposition}
    Suppose $(u_n,v_n)$ is the iteration sequence generated by Algorithm \ref{alg_uv} 
    to solve \eqref{kr_dis} with $K=I$ and $R(\cdot)=|\cdot|$ which is the discretized functional associated to \eqref{kr_new_rel}, then $F_{\alpha,\lambda}(u_n,v_n)$ converges to the unique minimum of $F_{\alpha,\lambda}$ on $V \times V$.
\end{proposition}

\begin{proof}
    As we solve problems \eqref{pro1} and \eqref{pro2} alternatively, we have
    \begin{equation}
        F_{\alpha,\lambda}(u_n,v_n) \geq F_{\alpha,\lambda}(u_n,v_{n+1}) \geq F_{\alpha,\lambda}(u_{n+1},v_{n+1}).
    \end{equation}
    The sequence $F_{\alpha,\lambda}(u_n,v_n)$ is non-increasing and bounded from below by 0, it thus converges to a constant $m$ in $\mathbb{R}$.
    As $F_{\alpha,\lambda}(u,v)$ is coercive, we can extract a subsequence $(u_{n_k},v_{n_k})$ which converges to $(\hat{u}, \hat{v})$ as $n_k \longrightarrow \infty$.
    Moreover, we have for all $n_k \in \mathbb{N}$,
    \begin{eqnarray}
         & F_{\alpha,\lambda}(u_{n_k},v_{n_k+1}) \leq F_{\alpha,\lambda}(u_{n_k},v) \quad \forall v \in V, \label{nk}\\
         & F_{\alpha,\lambda}(u_{n_k},v_{n_k}) \leq F_{\alpha,\lambda}(u,v_{n_k}) \quad \forall u \in V. \label{nk-1}
    \end{eqnarray}
    Denote $\tilde{v}$ as a cluster point of $v_{n_k+1}$, we have 
    \begin{equation}
        m = F_{\alpha,\lambda}(\hat{u},\hat{v}) = F_{\alpha,\lambda}(\hat{u},\tilde{v}). 
    \end{equation}
    Since \eqref{pro1} is strictly convex with respect to $v$, it has a unique minimizer for given $\hat{u}$.
    Hence $\tilde{v} = \hat{v}$, i.e. $\hat{v}$ is a cluster point of $v_{n_k+1}$. 
    By passing to the limit in \eqref{nk}, we have $F_{\alpha,\lambda}(\hat{u}, \hat{v}) = \inf_{v} F_{\alpha,\lambda}(\hat{u},v)$. 
    And by passing to the limit in \eqref{nk-1}, we have $F_{\alpha,\lambda}(\hat{u}, \hat{v}) = \inf_{u} F_{\alpha,\lambda}(u,\hat{v})$.
    Thus we have $(\hat{u}, \hat{v})$ is a critical point of $F_{\alpha,\lambda}$. 
    This is equivalent to $m = F_{\alpha,\lambda}(\hat{u}, \hat{v}) = \inf_{(u,v)} F_{\alpha, \lambda}$.
    Hence the sequence $F_{\alpha,\lambda}(u_n,v_n)$ converges to the unique minimum of $F_{\alpha,\lambda}$.
    We deduce that the sequence $(u_n,v_n)$ converges to the unique minimizer.                 $\hfill\square$
\end{proof}

\section{Numerical Experiments}
\label{sec:2}

    In this section, we present numerical results by applying the proposed models for real images. We also compare them with other models, including the ROF model and $G$-TV model.

    We set the mesh sizes $h=1$ in discretization in Sect.~\ref{Discretization}. 
    The images should be rescaled in $[0, 1]$ for its intensity.

In Figure \ref{Fig:car_bar_2}, we show the decomposition of a given image into cartoon and texture components using the ROF model, $G$-TV model and the proposed model.
For the $G$-TV model \eqref{g-pro-l2} and our model \eqref{kr_new_rel}, we set $\alpha=200$ and the other parameters are adjusted such that the cartoon parts have the same total variations as the ROF model \eqref{rof}.
Note the loss of intensity on the face area for the ROF model.
The $G$-TV model is an improvement of the ROF model to decompose an image \cite{2006Structure}.
The texture parts in the scarf are better captured in the oscillatory component $v$ and the structure of the face is kept in the cartoon part by using our model. 
\begin{figure}
    \centering
    \subfigure[$f$]{
        \begin{minipage}{0.3\linewidth}
        \centering
        \includegraphics[width= \linewidth]{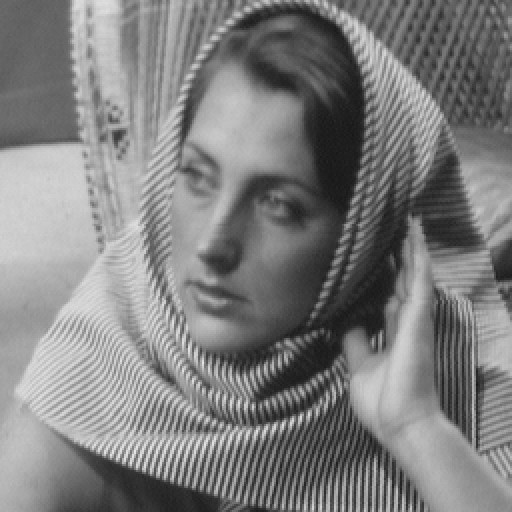}
    \end{minipage}
    }
    
    \subfigure[$u$ by ROF]{
        \begin{minipage}{0.3\linewidth}
        \centering
        \includegraphics[width= \linewidth]{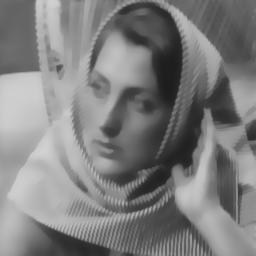}
    \end{minipage}
    }
    \subfigure[$u$ by $G$-TV]{
        \begin{minipage}{0.3\linewidth}
        \centering
        \includegraphics[width= \linewidth]{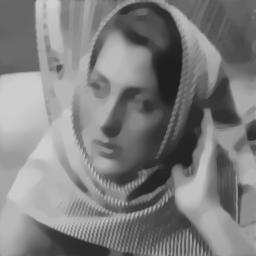}
    \end{minipage}
    }
    \subfigure[$u$ by our model]{
        \begin{minipage}{0.3\linewidth}
        \centering
        \includegraphics[width=\linewidth]{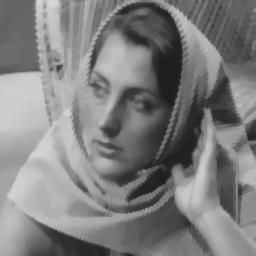}
    \end{minipage}
    }
    
    \subfigure[$f-u+100/255$]{
        \begin{minipage}{0.3\linewidth}
        \centering
        \includegraphics[width= \linewidth]{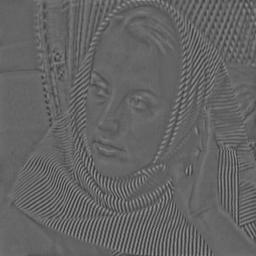}
    \end{minipage}
    } 
    \subfigure[$f-u+100/255$]{
        \begin{minipage}{0.3\linewidth}
        \centering
        \includegraphics[width= \linewidth]{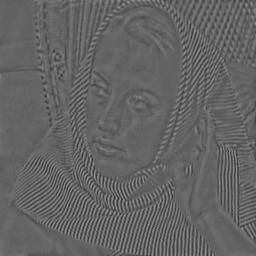}
    \end{minipage}
    }
    \subfigure[$f-u+100/255$]{
        \begin{minipage}{0.3\linewidth}
        \centering
        \includegraphics[width=\linewidth]{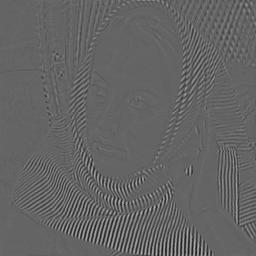}
    \end{minipage}
    }
    
    \caption{The cartoon and texture decomposition $f=u+v$ obtained from the ROF, $G$-TV \eqref{g-pro-l2}, and our model \eqref{kr_new_rel}. 
    }
    \label{Fig:car_bar_2}
\end{figure}

    In the following, we apply \eqref{kr_new_rel} to denoising real images.
Moreover, we extend our model by employing a variant of total variation developed in \cite{zhu} in order to preserve image contrast and suppress the staircase phenomenon.
We compare our model with the modified TV (MTV) \eqref{kr_new_rel_log} with the original MTV model \eqref{log-tv} in \cite{zhu}.

The noisy image is corrupted by Gaussian noise with standard deviation $\sigma$, i.e. the noise $n \sim \mathcal{N}(0,\sigma^2)$.
To be fair for the performance comparison among these proposed models, we choose the regularization parameters such that the norm of the removed noise $\|f-u\|$, is almost the same for each model. 
Specifically, the averaged Frobenius norm $\|f-u\|_F = \sqrt{\|f-u\|/|\Omega|} $ is of the same size as the noise level $\sigma$ for each model, as in \cite{1992Nonlinear,zhu,bregman,osher2003image,2020zhuLp}.

To compare the quality of the results obtained by these models, we list in Table~\ref{table:psnr_1} the peak signal-to-noise ratios (PSNRs) for the image denoising experiments in this paper.
In this test, for the fair comparison, we set up the regularization parameter for each model such that the quantity $\|f-u\| \approx N\sigma$.
\begin{table}
\caption{The values of PSNR for the restored images.}
\label{table:psnr_1}        
\begin{tabular}{lllll}
\hline\noalign{\smallskip}
\textbf{Algorithm}                    & \textbf{ROF}  & \textbf{MTV} & \textbf{Our model} & \textbf{Our model with MTV}   \\   
    \hline Figure \ref{Fig:u_n_house} & 30.04         & 29.46            & 30.46             & 30.18  \\ 
    \hline  Figure \ref{Fig:u_n_but}  & 26.95          &  27.18              & 27.49          & 27.31 \\ 
    \hline Figure \ref{Fig:u_n_par}  &  28.24        & 28.13         & 28.13     & 28.41  \\ 
    \hline Figure \ref{Fig:man_log}  &  30.12        & 30.09          & 30.23      & 30.26  \\ 
\noalign{\smallskip}\hline
\end{tabular}
\end{table}

In Figure \ref{Fig:u_n_house}, we consider a non-texture image, "House", corrupted by a Gaussian noise with zero mean and standard deviation or level $\sigma=25/255$.   
Though the ROF model can efficiently eliminate the noise, it suffers from the loss of image contrast.
As shown in the residual images $f-u+100/255$, many meaningful signals are removed as noise by the ROF model, while much fewer signals are swept as noise by our model \eqref{kr_new_rel}.
This example demonstrates that our model is capable of preserving image contrast.
Similar to the $G$-TV model \cite{2005Image}, the component $v = \mathrm{div}(m)$ of the proposed models have a zero mean which is the mean of the added white Gaussian noise.
The loss of intensity property is always present for the ROF model. The white regions depicted in images restored by the proposed models are brighter and the dark regions are darker than those of the ROF model.

Moreover, the proposed model with MTV \eqref{kr_new_rel_log} preserves image contrasts better than the MTV model \eqref{log-tv}, and suppresses the staircase effect compared with our model with TV. 
The data in Table~\ref{table:psnr_1} show that our model produces a higher PSNR than the ROF model.

\begin{figure}
    \centering
    \subfigure[Original Image]{
        \begin{minipage}{0.23\linewidth}
        \centering
        \includegraphics[width= \linewidth]{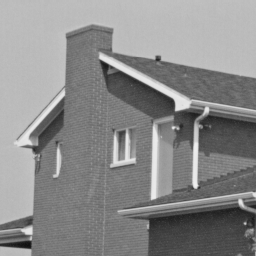}
    \end{minipage}
    }  
    \subfigure[Noisy Image]{
        \begin{minipage}{0.23\linewidth}
        \centering
        \includegraphics[width= \linewidth]{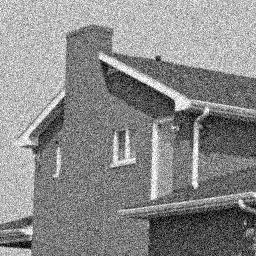}
    \end{minipage}
    } 
    
    \subfigure[$u$ by ROF]{
        \begin{minipage}{0.23\linewidth}
        \centering
        \includegraphics[width= \linewidth]{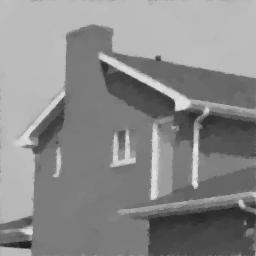}
    \end{minipage}
    }
    \subfigure[$f-u+100/255$]{
        \begin{minipage}{0.23\linewidth}
        \centering
        \includegraphics[width= \linewidth]{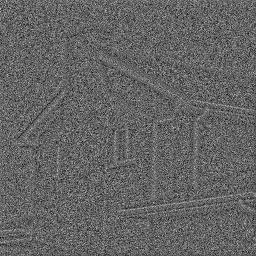}
    \end{minipage}
    }      
    \subfigure[$u$ by our model]{
        \begin{minipage}{0.23\linewidth}
        \centering
        \includegraphics[width= \linewidth]{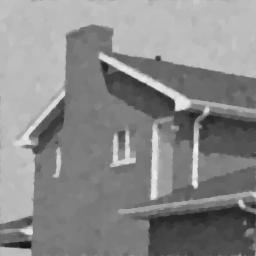}
    \end{minipage}
    }
    \subfigure[$f-u+100/255$]{
        \begin{minipage}{0.23\linewidth}
        \centering
        \includegraphics[width=\linewidth]{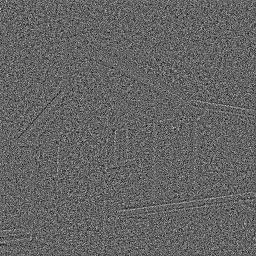}
    \end{minipage}
    }

    \subfigure[$u$ by MTV]{
        \begin{minipage}{0.23\linewidth}
        \centering
        \includegraphics[width= \linewidth]{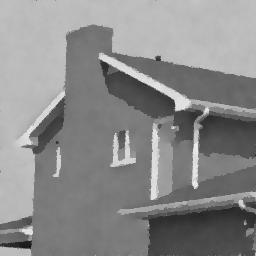}
    \end{minipage}
    }
    \subfigure[$f-u+100/255$]{
        \begin{minipage}{0.23\linewidth}
        \centering
        \includegraphics[width=\linewidth]{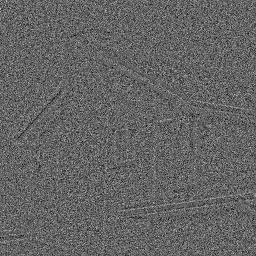}
    \end{minipage}
    }  
    \subfigure[$u$ by our model with MTV]{
        \begin{minipage}{0.23\linewidth}
        \centering
        \includegraphics[width= \linewidth]{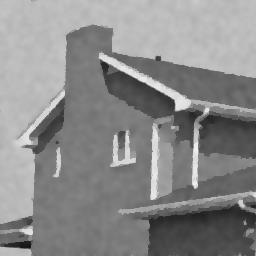}
    \end{minipage}
    }
    \subfigure[$f-u+100/255$]{
        \begin{minipage}{0.23\linewidth}
        \centering
        \includegraphics[width= \linewidth]{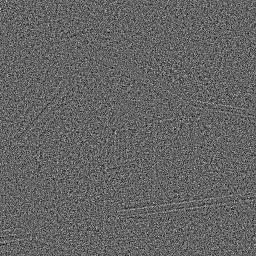}
    \end{minipage}
    }

    \caption{Results using the ROF model, our model \eqref{kr_new_rel}, MTV ($a=5/255$) \eqref{log-tv}, and the proposed model with MTV \eqref{kr_new_rel_log}. 
    The additive noise $n$ has the norm $\|n\| = 25$.
    For the proposed models \eqref{kr_new_rel} and \eqref{kr_new_rel_log}, we choose parameter $\alpha=100$, and the other parameters are tuned so that the removed noise parts have the same norm $\|f-u\| \approx 25$ as the ROF model.
    }
    \label{Fig:u_n_house}
\end{figure}

In Figure \ref{Fig:u_n_but}, we consider an image "Butterfly", corrupted by Gaussian noise with level $\sigma=25/255$.
The residual images show that our model sweeps fewer meaningful signals as noise than the ROF model, especially for the large scale part, which suggests that our model is able to preserve image contrast.
The proposed model with MTV \eqref{kr_new_rel_log} keeps some fine details in the restored images. 
Table~\ref{table:psnr_1} shows that the PSNR for our model is higher than that of the ROF model.
\begin{figure}
    \centering
    \subfigure[Original Image]{
        \begin{minipage}{0.23\linewidth}
        \centering
        \includegraphics[width= \linewidth]{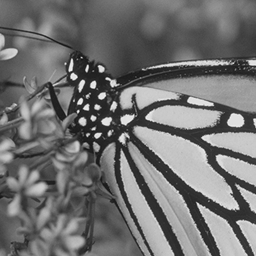}
    \end{minipage}
    }
    \subfigure[Noisy Image]{
        \begin{minipage}{0.23\linewidth}
        \centering
        \includegraphics[width= \linewidth]{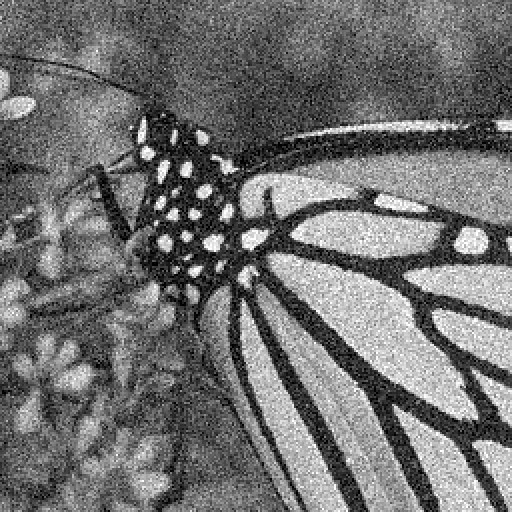}
    \end{minipage}
    }
    
    \subfigure[$u$ by ROF]{
        \begin{minipage}{0.23\linewidth}
        \centering
        \includegraphics[width= \linewidth]{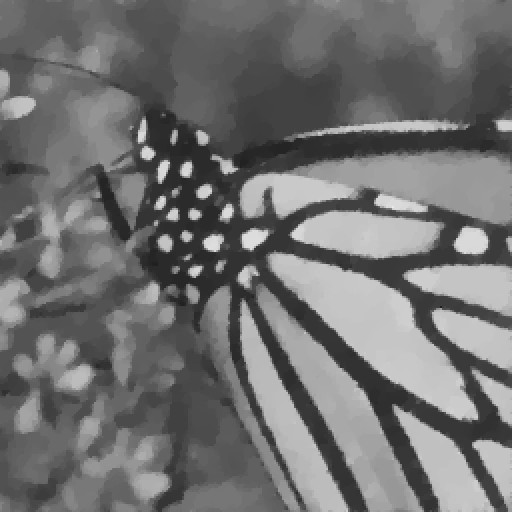}
    \end{minipage}
    }
    \subfigure[$f-u+100/255$]{
        \begin{minipage}{0.23\linewidth}
        \centering
        \includegraphics[width= \linewidth]{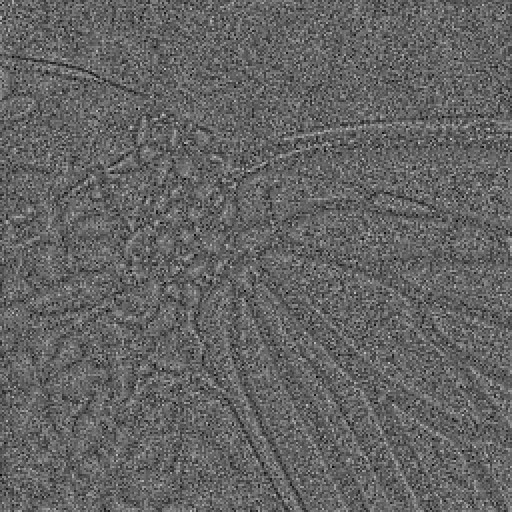}
    \end{minipage}
    }
    \subfigure[$u$ by our model]{
        \begin{minipage}{0.23\linewidth}
        \centering
        \includegraphics[width= \linewidth]{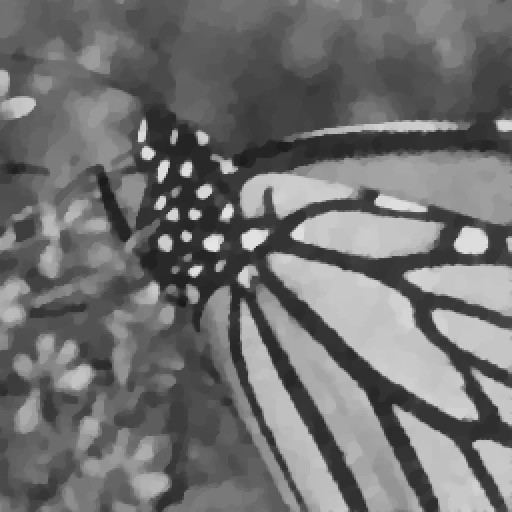}
    \end{minipage}
    }
    \subfigure[$f-u+100/255$]{
        \begin{minipage}{0.23\linewidth}
        \centering
        \includegraphics[width=\linewidth]{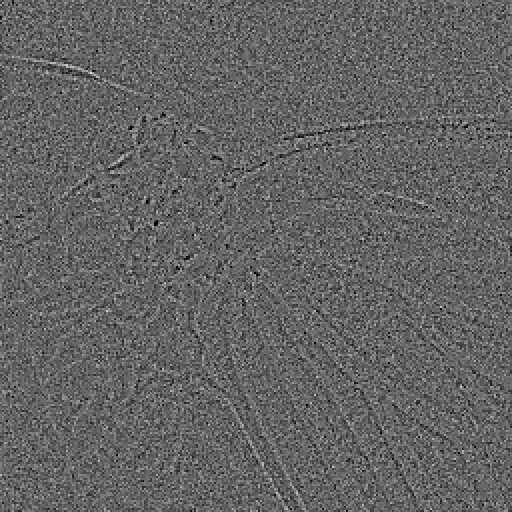}
    \end{minipage}
    }

    \subfigure[$u$ by MTV]{
        \begin{minipage}{0.23\linewidth}
        \centering
        \includegraphics[width= \linewidth]{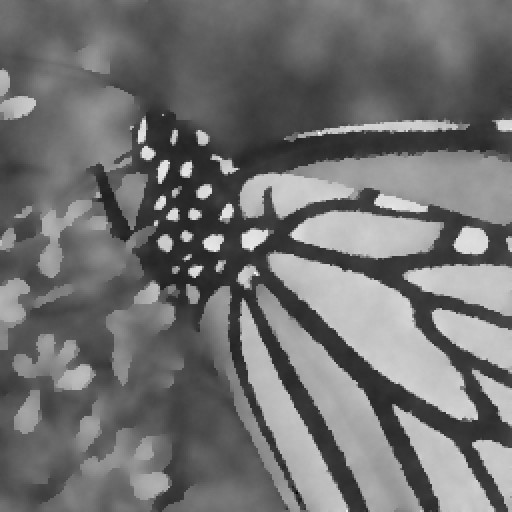}
    \end{minipage}
    }
    \subfigure[$f-u+100/255$]{
        \begin{minipage}{0.23\linewidth}
        \centering
        \includegraphics[width= \linewidth]{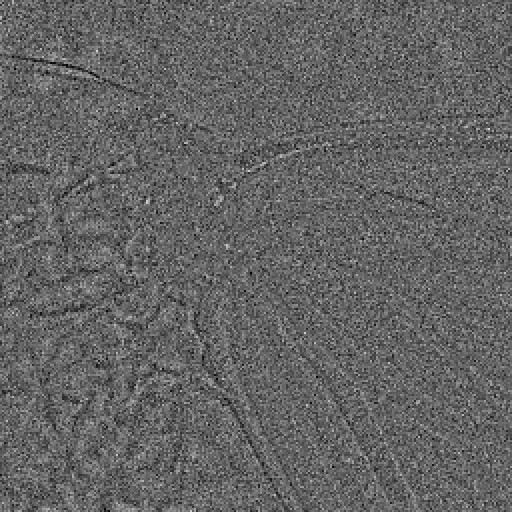}
    \end{minipage}
    }
    \subfigure[$u$ by our model with MTV]{
        \begin{minipage}{0.23\linewidth}
        \centering
        \includegraphics[width= \linewidth]{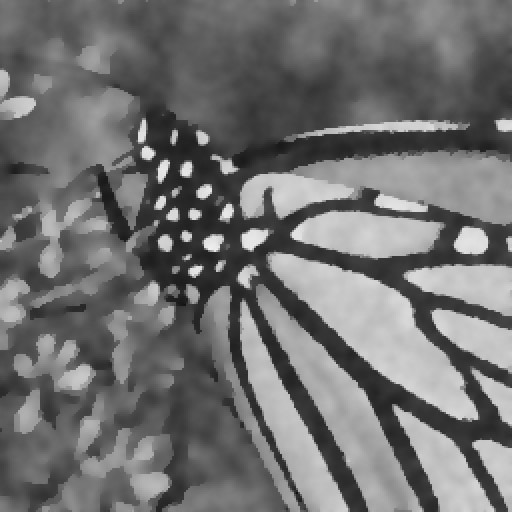}
    \end{minipage}
    }
    \subfigure[$f-u+100/255$]{
        \begin{minipage}{0.23\linewidth}
        \centering
        \includegraphics[width=\linewidth]{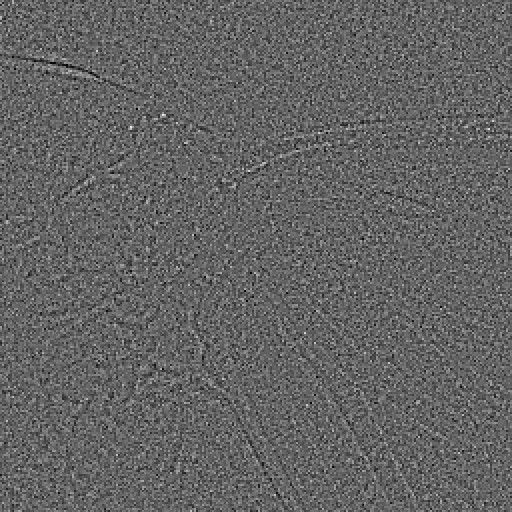}
    \end{minipage}
    }
    
    \caption{Results using the ROF model, our model \eqref{kr_new_rel}, MTV model ($a=8/255$)\eqref{log-tv} and the proposed model with MTV \eqref{kr_new_rel_log}.
    The additive noise $n$ has norm $\|n\| = 25$.
    For the proposed models, we choose the same parameter $\alpha=80$ and the other parameters are tuned such that the removed noise part has the norm $\|f-u\| \approx 25$ as the ROF model and MTV model.}
    \label{Fig:u_n_but}
\end{figure}

In Figure \ref{Fig:u_n_par}, we consider an image "Parrot". 
The ROF model suffers from the loss of image contrasts and the staircase effect.
As shown in residual images $f-u+100/255$, our model sweeps less meaningful signals as noise than the ROF model in the large scale part.
Our model with MTV keeps some fine textures in the restored images in region with rich textures, and also suppresses staircasing effect.
Table~\ref{table:psnr_1} shows that the PSNR for our model is not optimal, but our model with MTV produces a higher PSNR than the ROF model.

\begin{figure}
    \centering
    \subfigure[Original Image]{
        \begin{minipage}{0.23\linewidth}
        \centering
        \includegraphics[width= \linewidth]{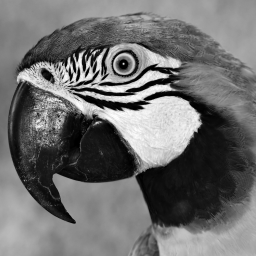}
    \end{minipage}
    }
    \subfigure[Noisy Image]{
        \begin{minipage}{0.23\linewidth}
        \centering
        \includegraphics[width= \linewidth]{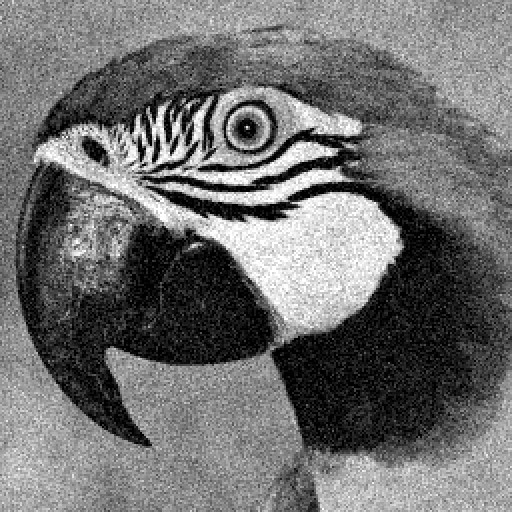}
    \end{minipage}
    }
    
    \subfigure[$u$ by ROF]{
        \begin{minipage}{0.23\linewidth}
        \centering
        \includegraphics[width= \linewidth]{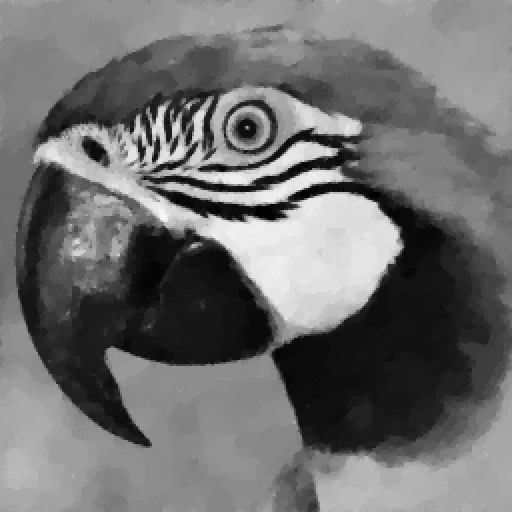}
    \end{minipage}
    }
    \subfigure[$f-u+100/255$]{
        \begin{minipage}{0.23\linewidth}
        \centering
        \includegraphics[width= \linewidth]{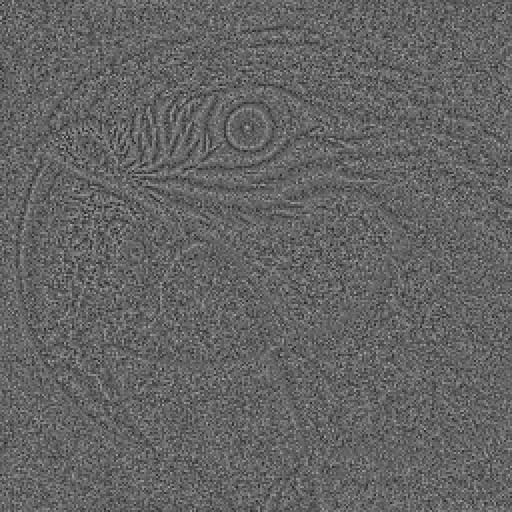}
    \end{minipage}
    }
    \subfigure[$u$ by our model]{
        \begin{minipage}{0.23\linewidth}
        \centering
        \includegraphics[width=\linewidth]{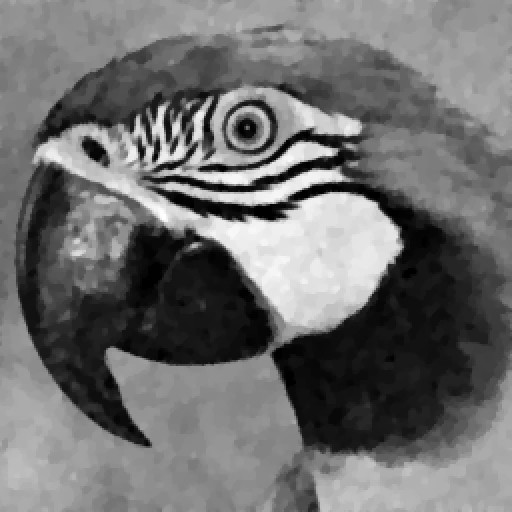}
    \end{minipage}
    }
    \subfigure[$f-u+100/255$]{
        \begin{minipage}{0.23\linewidth}
        \centering
        \includegraphics[width=\linewidth]{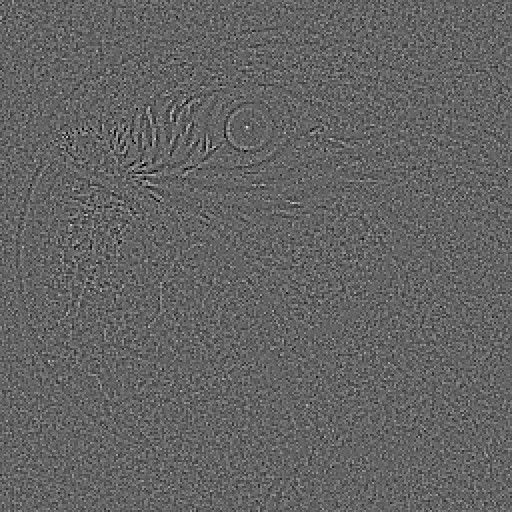}
    \end{minipage}
    }

    \subfigure[$u$ by MTV]{
        \begin{minipage}{0.23\linewidth}
        \centering
        \includegraphics[width= \linewidth]{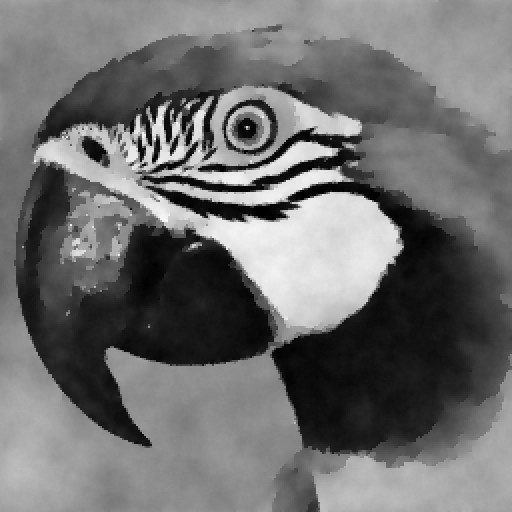}
    \end{minipage}
    }
    \subfigure[$f-u+100/255$]{
        \begin{minipage}{0.23\linewidth}
        \centering
        \includegraphics[width= \linewidth]{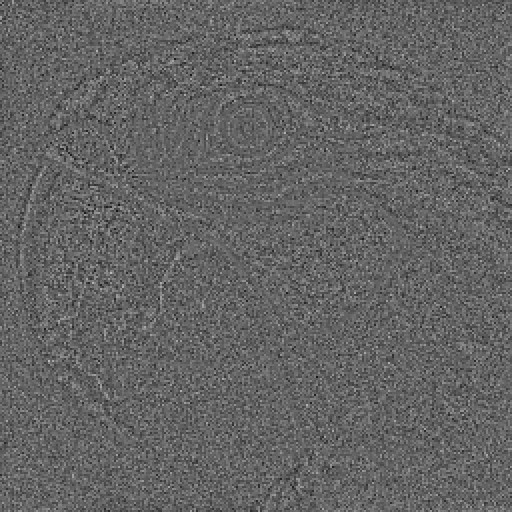}
    \end{minipage}
    }
    \subfigure[$u$ by our model with MTV]{
        \begin{minipage}{0.23\linewidth}
        \centering
        \includegraphics[width= \linewidth]{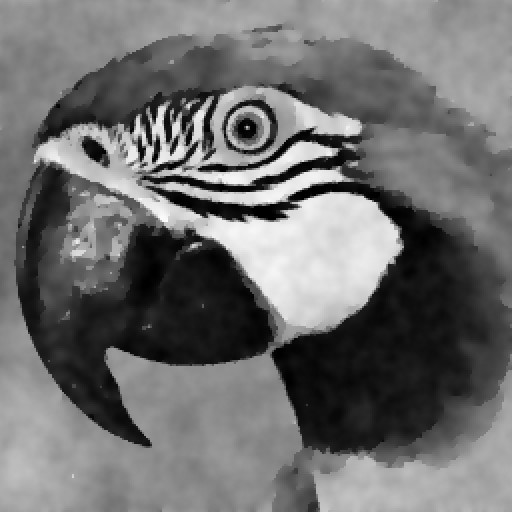}
    \end{minipage}
    }
    \subfigure[$f-u+100/255$]{
        \begin{minipage}{0.23\linewidth}
        \centering
        \includegraphics[width=\linewidth]{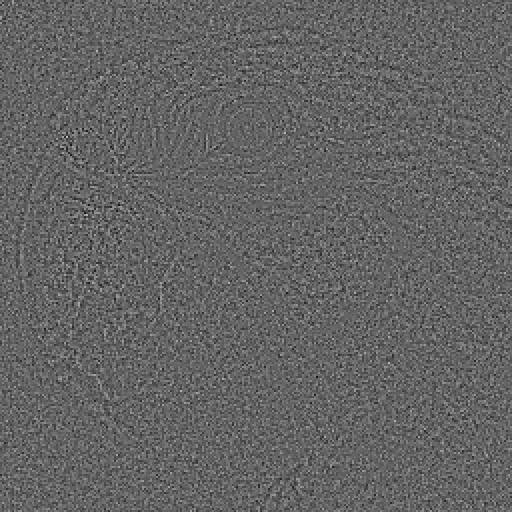}
    \end{minipage}
    }
    
    \caption{Results using the ROF model, our model \eqref{kr_new_rel}, MTV ($a=8/255$) \eqref{log-tv} and the proposed model with MTV \eqref{kr_new_rel_log}. 
    The additive noise $n$ has the norm $\|n\| = 20$ and the parameters are tuned such that the removed noise part has the norm $\|f-u\| \approx 20$ for all the models.
    }
    \label{Fig:u_n_par}
\end{figure}

    In Figure \ref{Fig:star}, we apply the proposed models to restore noisy and blurry images. 
The proposed model leaves less information in the residual image than the ROF model, and the restored image $u$ by using the proposed model appears to be sharper.
\begin{figure}
    \centering
    \subfigure[Original image]{
        \begin{minipage}{0.3\linewidth}
        \centering
        \includegraphics[width= \linewidth]{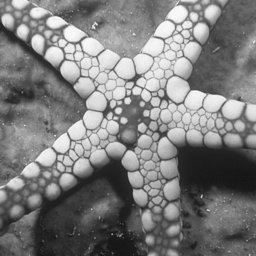}
    \end{minipage}
    }
    \subfigure[$u$ by ROF]{
        \begin{minipage}{0.3\linewidth}
        \centering
        \includegraphics[width= \linewidth]{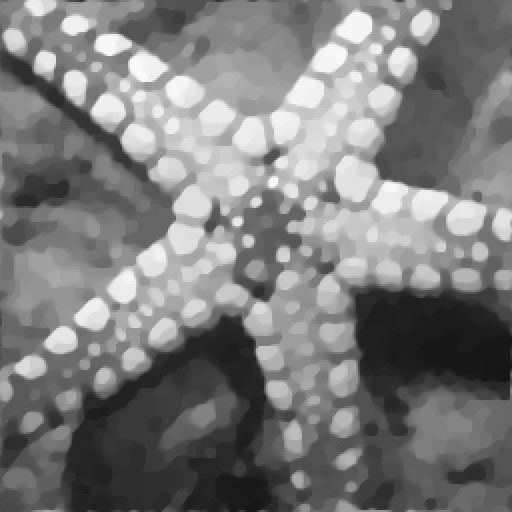}
    \end{minipage}
    }
    \subfigure[$u$ by our model]{
        \begin{minipage}{0.3\linewidth}
        \centering
        \includegraphics[width= \linewidth]{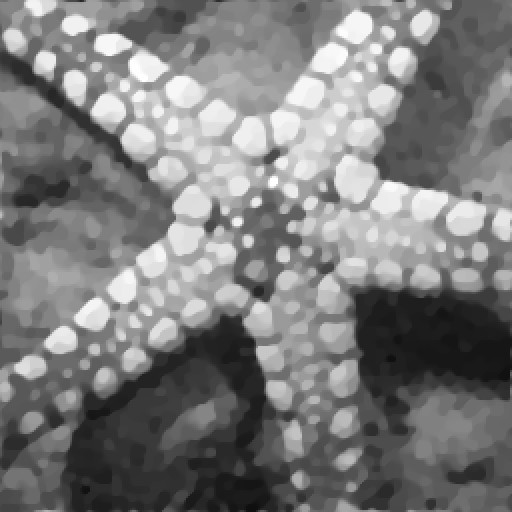}
    \end{minipage}
    } 
    
    \subfigure[Noisy and blurred image]{
        \begin{minipage}{0.3\linewidth}
        \centering
        \includegraphics[width= \linewidth]{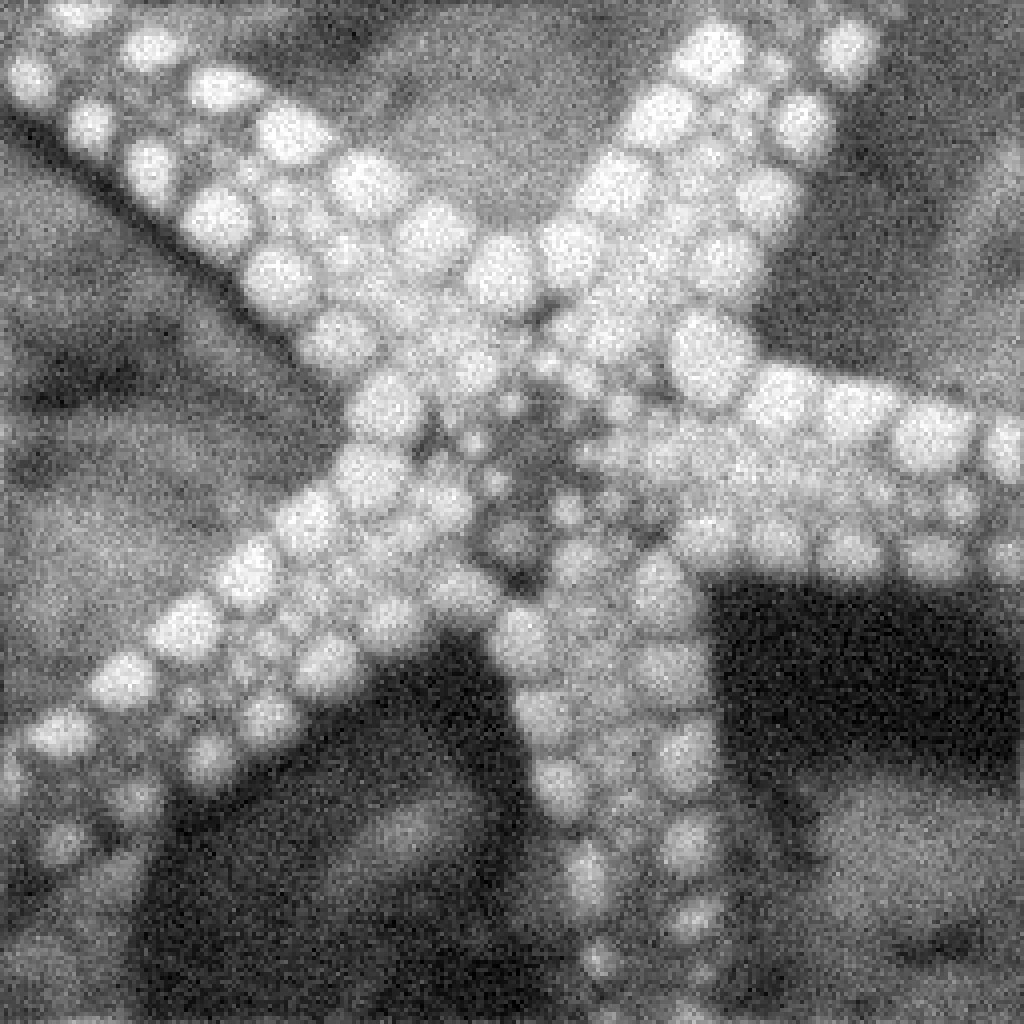}
    \end{minipage}
    }
    \subfigure[$f-K*u+100/255$]{
        \begin{minipage}{0.3\linewidth}
        \centering
        \includegraphics[width= \linewidth]{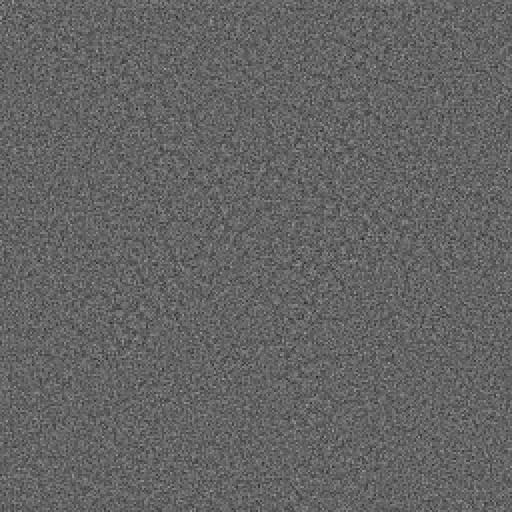}
    \end{minipage}
    }
    \subfigure[$f-K*u+100/255$]{
        \begin{minipage}{0.3\linewidth}
        \centering
        \includegraphics[width= \linewidth]{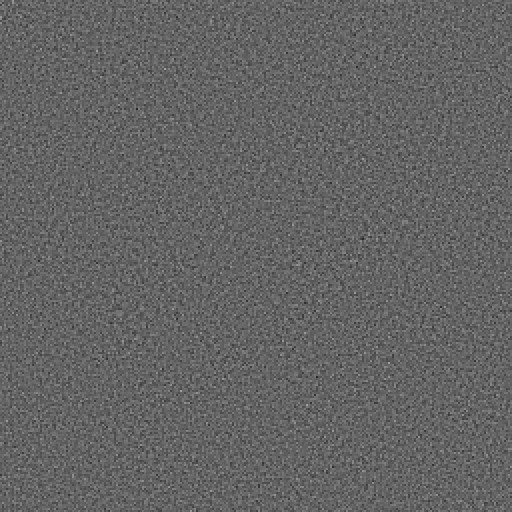}
    \end{minipage}
    }

    \caption{Results using the ROF model, and our model \eqref{kr_new_rel}. The noisy and blurred image is obtained by applying a Gaussian blur $(G,1.6,25)$ and adding a Gaussian noise with $\sigma=15$.
    The additive noise $n$ has the norm $\|n\|_{L^2} \approx 15$. The parameters are adjusted such that the removed noise parts of all the models have the norm $\|f-K*u\|_{L^2} \approx 15$.}
    \label{Fig:star}
\end{figure}

To monitor the convergence of the numerical method in Sect~\ref{sec:pdhg} for the subproblem \eqref{pro1}, we check the fixed point residual $R^k$ \eqref{residual} and the constraint in \eqref{vpro}:
\begin{equation}\label{cons}
    \frac{\|v^k - \mathrm{div}(m^k)\|}{\|v^k-f\|}.
\end{equation}
The value of $R^k$ \eqref{residual} can be used as a stop criterion to terminate the iterative steps \eqref{pdhg-kr} and we usually set $\epsilon=O(h^2)$.
In Figure \ref{Fig:pdhg}, we present the plots of $R^k$ and the constraint versus iteration for the subproblem \eqref{pro1} in Algorithm \ref{alg_uv} for the example in Figure \ref{Fig:u_n_house}.

\begin{figure}
    \centering
    \subfigure{
        \begin{minipage}{0.4\linewidth}
        \centering
        \includegraphics[width= \linewidth]{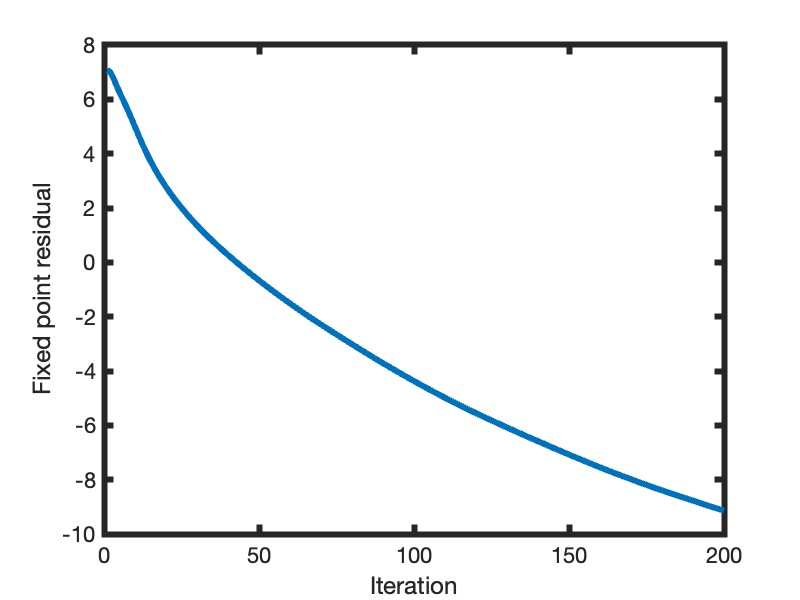}
    \end{minipage}
    }
    \subfigure{
        \begin{minipage}{0.4\linewidth}
        \centering
        \includegraphics[width= \linewidth]{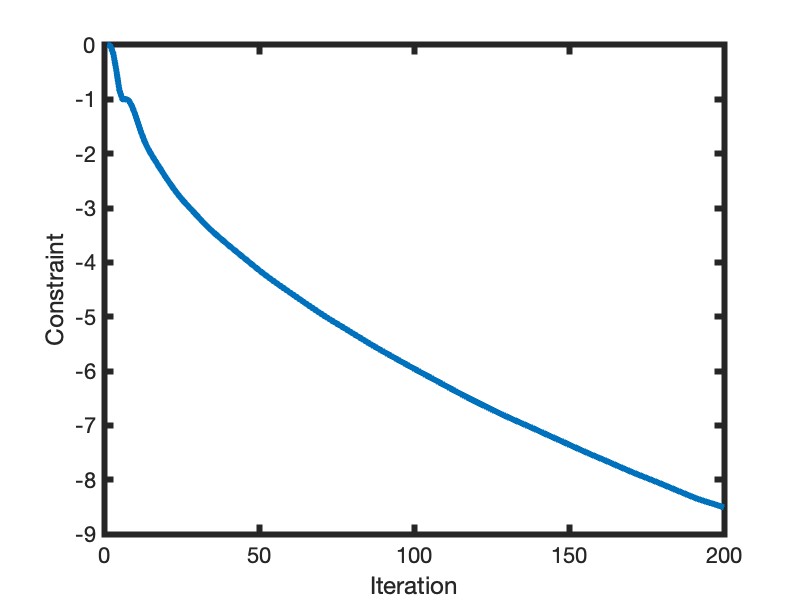}
    \end{minipage}
    }
    \caption{The plots of fixed point residual $R^k$ \eqref{residual} and the constraint \eqref{cons} versus iteration.
    All the above quantities are shown in \textit{log-scale}.}
    \label{Fig:pdhg}
\end{figure}

To monitor the convergence of Algorithm \ref{alg_uv}, we check the relative error of the solution $u^n$ and $v^n$,
    \begin{equation}\label{rel-uv}
        \frac{\|u^n-u^{n-1}\|}{\|u^{n-1}\|},\quad \frac{\|v^n-v^{n-1}\|}{\|v^{n-1}\|}.
    \end{equation}
In Figure \ref{Fig:alg2}, we list the plots of them versus iteration by applying Algorithm \ref{alg_uv} for the experiment in Figure \ref{Fig:u_n_house}.

\begin{figure}
    \centering
    \subfigure{
        \begin{minipage}{0.4\linewidth}
        \centering
        \includegraphics[width= \linewidth]{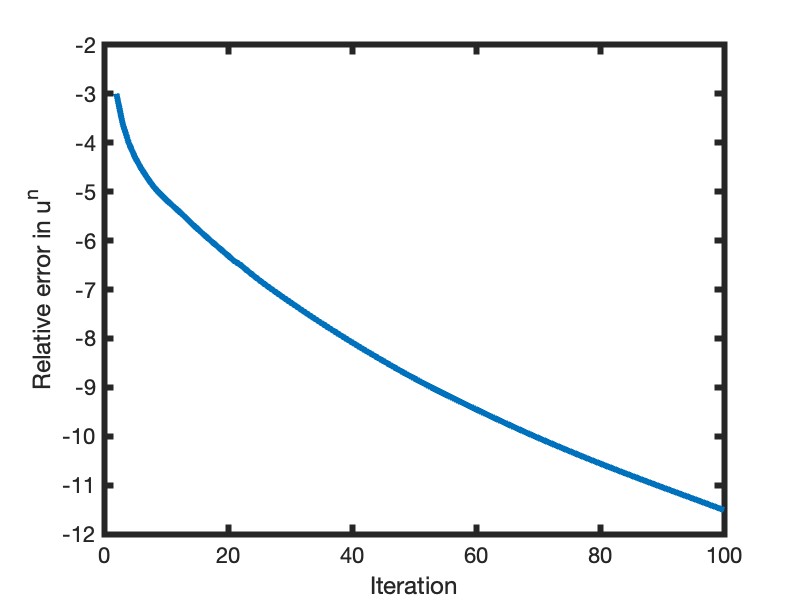}
    \end{minipage}
    }
    \subfigure{
        \begin{minipage}{0.4\linewidth}
        \centering
        \includegraphics[width= \linewidth]{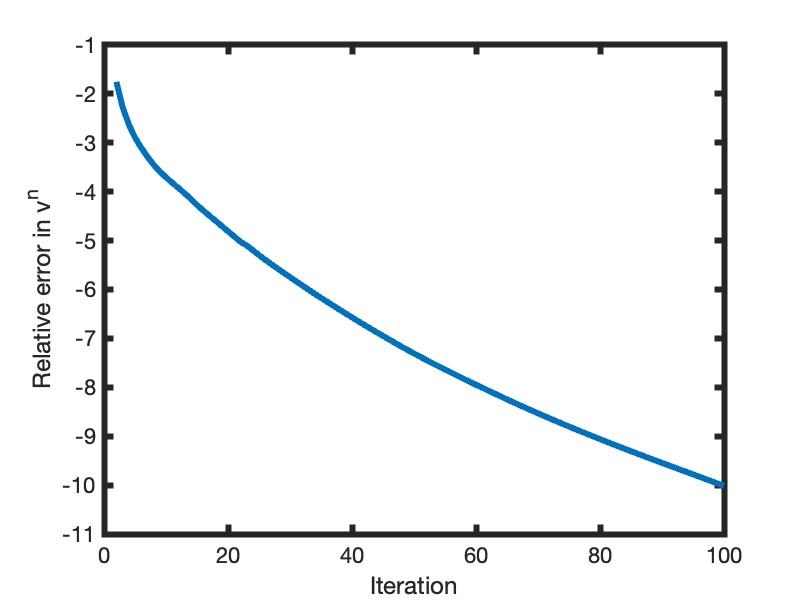}
    \end{minipage}
    }
    \caption{
    The plots of the relative error of $u^n$ and $v^n$ \eqref{rel-uv} versus iteration. All the above quantities are shown in \textit{log-scale}.}
    \label{Fig:alg2}
\end{figure}

\section{Conclusion}\label{sec:conclusion}

    In this paper, we propose image restoration models which combines the total variation regularization with the dual Lipschitz norm from optimal transport.
    To solve of the proposed models, we construct numerical method based on the Primal-Dual Hybrid Gradient algorithms for the calculation of the Wasserstain-1 distance and the augmented Lagrangian method for the total variation. The convergence analysis of the numerical method is presented.
    We prove the existence and uniquess of the minimizer of the associated functional and some additional theoretical results based on the relation between the dual Lipschitz norm and the $G$-norm. 
    Numerical results are presented to illustrate the features of the proposed models for image restoration and cartoon-texture decomposition, and also the efficacy of the designed numerical methods.
    Moreover, to preserve image contrast and suppress the staircase phenomenon better, we also consider replacing the total variation in our model by one of its modifications as in \cite{zhu}.


\begin{acknowledgements}
This work was partially supported by the NSFC Projects No. 12025104, 12001529.
\end{acknowledgements}

%
\section*{Conflict of interest}
The authors declare that they have no conflict of interest.


\bibliographystyle{spmpsci}
\bibliography{ref.bib}

\begin{thebibliography}{10}
\providecommand{\url}[1]{{#1}}
\providecommand{\urlprefix}{URL }
\expandafter\ifx\csname urlstyle\endcsname\relax
  \providecommand{\doi}[1]{DOI~\discretionary{}{}{}#1}\else
  \providecommand{\doi}{DOI~\discretionary{}{}{}\begingroup \urlstyle{rm}\Url}\fi

\bibitem{2017Wasserstein}
Arjovsky, M., Chintala, S., Bottou, L.: Wasserstein generative adversarial networks.
\newblock In: Proc. Mach. Learn. Res., pp. 214--223. PMLR (2017)

\bibitem{2005Image}
Aujol, J.F., Aubert, G., Blanc-F{\'e}raud, L., Chambolle, A.: Image decomposition into a bounded variation component and an oscillating component.
\newblock J. Math. Imaging Vision \textbf{22}, 71--88 (2005)

\bibitem{aujol2005dual}
Aujol, J.F., Chambolle, A.: Dual norms and image decomposition models.
\newblock Int. J. Comput. Vision \textbf{63}, 85--104 (2005)

\bibitem{2006Structure}
Aujol, J.F., Gilboa, G., Chan, T., Osher, S.: Structure-texture image decomposition—modeling, algorithms, and parameter selection.
\newblock Int. J. Comput. Vision \textbf{67}(1), 111--136 (2006)

\bibitem{2009A}
Beck, A., Teboulle, M.: A fast iterative shrinkage-thresholding algorithm for linear inverse problems.
\newblock SIAM J. Imaging Sci. \textbf{2}, 183--202 (2009)

\bibitem{2000A}
Benamou, J.D., Brenier, Y.: A computational fluid mechanics solution to the monge-kantorovich mass transfer problem.
\newblock Numer. Math. \textbf{84}(3), 375--393 (2000)

\bibitem{2014Iterative}
Benamou, J.D., Carlier, G., Cuturi, M., Nenna, L., Peyré, G.: Iterative bregman projections for regularized transportation problems.
\newblock SIAM J. Sci. Comput. \textbf{37}(2), A1111--A1138 (2014)

\bibitem{2015Sliced}
Bonneel, N., Rabin, J., Peyré, G., Pfister, H.: Sliced and radon wasserstein barycenters of measures.
\newblock J. Math. Imaging Vision \textbf{51}(1), 22--45 (2015)

\bibitem{Bredies2010TotalGV}
Bredies, K., Kunisch, K., Pock, T.: Total generalized variation.
\newblock SIAM J. Imaging Sci. \textbf{3}, 492--526 (2010).
\newblock \urlprefix\url{https://api.semanticscholar.org/CorpusID:6650697}

\bibitem{2012Regularized}
Burger, M., Franek, M., Sch{\"o}nlieb, C.B.: Regularized regression and density estimation based on optimal transport.
\newblock Appl. Math. Res. eXpress \textbf{2012}(2), 209--253 (2012)

\bibitem{cai2012image}
Cai, J.F., Dong, B., Osher, S., Shen, Z.: Image restoration: total variation, wavelet frames, and beyond.
\newblock J. Am. Math. Soc. \textbf{25}(4), 1033--1089 (2012)

\bibitem{cai2010}
Cai, J.F., Osher, S., Shen, Z.: Split bregman methods and frame based image restoration.
\newblock Multiscale Model. Simul. \textbf{8}(2), 337--369 (2010).
\newblock \doi{10.1137/090753504}

\bibitem{2004cha}
Chambolle, A.: An algorithm for total variation minimization and applications.
\newblock J. Math. Imaging Vision \textbf{20}, 89--97 (2004)

\bibitem{2011A}
Chambolle, A., Pock, T.: A first-order primal-dual algorithm for convex problems with applications to imaging.
\newblock J. Math. Imaging Vision \textbf{40}, 120--145 (2011)

\bibitem{Chizat2018Scaling}
Chizat, L., Peyré, G., Schmitzer, B., Vialard, F.X.: Scaling algorithms for unbalanced optimal transport problems.
\newblock Math. Comp. \textbf{87}(314), 2563--2609 (2018)

\bibitem{cuturi2013sinkhorn}
Cuturi, M.: Sinkhorn distances: Lightspeed computation of optimal transport.
\newblock Advances in neural information processing systems \textbf{26} (2013)

\bibitem{2015A}
Cuturi, M., Peyr{\'e}, G.: A smoothed dual approach for variational wasserstein problems.
\newblock SIAM J. Imaging Sci. \textbf{9}(1), 320--343 (2016)

\bibitem{2012Wavelet}
Dong, B., Ji, H., Li, J., Shen, Z., Xu, Y.: Wavelet frame based blind image inpainting.
\newblock Appl. Comput. Harmon. Anal. \textbf{32}(2), 268--279 (2012)

\bibitem{2019Unnormalized}
Gangbo, W., Li, W., Osher, S., Puthawala, M.: Unnormalized optimal transport.
\newblock J. Comput. Phys. \textbf{399}, 108940 (2019)

\bibitem{Gao_2023}
Gao, Y., Jin, Z., Li, X.: Template-based ct reconstruction with optimal transport and total generalized variation.
\newblock Inverse Problems \textbf{39}(9), 095007 (2023).
\newblock \doi{10.1088/1361-6420/aceb17}.
\newblock \urlprefix\url{https://dx.doi.org/10.1088/1361-6420/aceb17}

\bibitem{garnett2011modeling}
Garnett, J.B., Jones, P.W., Le, T.M., Vese, L.A.: Modeling oscillatory components with the homogeneous spaces bm-$\alpha$ and w-$\alpha$, p.
\newblock Pure Appl. Math. Q. \textbf{7}(2), 275--318 (2011)

\bibitem{2007GarImage}
Garnett, J.B., Le, T.M., Meyer, Y., Vese, L.A.: Image decompositions using bounded variation and generalized homogeneous besov spaces.
\newblock Appl. Comput. Harmon. Anal. \textbf{23}(1), 25--56 (2007)

\bibitem{2004Optimal}
Haker, S., Zhu, L., Tannenbaum, A., Angenent, S.: Optimal mass transport for registration and warping.
\newblock Int. J. Comput. Vision \textbf{60}, 225--240 (2004)

\bibitem{2012He}
He, B., Tao, M., Yuan, X.: Alternating direction method with gaussian back substitution for separable convex programming.
\newblock SIAM Journal on Optimization \textbf{22}(2), 313--340 (2012)

\bibitem{2012Convergence}
He, B., Yuan, X.: Convergence analysis of primal-dual algorithms for a saddle-point problem: From contraction perspective.
\newblock SIAM J. Imaging Sci. \textbf{5}(1), 119--149 (2012)

\bibitem{2018Solving}
Jacobs, M., L{\'e}ger, F., Li, W., Osher, S.: Solving large-scale optimization problems with a convergence rate independent of grid size.
\newblock SIAM J. Numer. Anal. \textbf{57}(3), 1100--1123 (2019)

\bibitem{doi:10.1137/140967416}
Jung, M., Kang, M.: Simultaneous cartoon and texture image restoration with higher-order regularization.
\newblock SIAM J. Imaging Sci. \textbf{8}(1), 721--756 (2015).
\newblock \doi{10.1137/140967416}

\bibitem{2016Generalized}
Karlsson, J., Ringh, A.: Generalized sinkhorn iterations for regularizing inverse problems using optimal mass transport.
\newblock SIAM J. Imaging Sci. \textbf{10}(4), 1935--1962 (2017)

\bibitem{kim2009image}
Kim, Y., Vese, L.: Image recovery using functions of bounded variation and sobolev spaces of negative differentiability.
\newblock Inverse Probl. Imaging \textbf{3}(1), 43--68 (2009)

\bibitem{2017Optimal}
Kolouri, S., Park, S.R., Thorpe, M., Slepcev, D., Rohde, G.K.: Optimal mass transport: Signal processing and machine-learning applications.
\newblock IEEE Signal Process Mag. \textbf{34}(4), 43--59 (2017)

\bibitem{2005ImageBMO}
Le, T.M., Vese, L.A.: Image decomposition using total variation and div(bmo)*.
\newblock Multiscale Model. Simul. \textbf{4}(2), 390--423 (2005)

\bibitem{2014Imaging}
Lellmann, J., Lorenz, D., Sch{\"o}nlieb, C., Valkonen, T.: Imaging with kantorovich-rubinstein discrepancy.
\newblock SIAM J. Imaging Sci. \textbf{7}(4), 2833--2859 (2014)

\bibitem{2016A}
Li, W., Osher, S., Gangbo, W.: A fast algorithm for earth mover’s distance based on optimal transport and l1 type regularization.
\newblock arXiv \textbf{1609} (2016)

\bibitem{lieu2008image}
Lieu, L.H., Vese, L.A.: Image restoration and decomposition via bounded total variation and negative hilbert-sobolev spaces.
\newblock Appl. Math. Optim. \textbf{58}, 167--193 (2008)

\bibitem{2021MULTILEVEL}
Liu, J., Yin, W., Li, W., Chow, Y.T.: Multilevel optimal transport: a fast approximation of wasserstein-1 distances.
\newblock SIAM J. Sci. Comput. \textbf{43}(1), A193--A220 (2021)

\bibitem{meyeroscillating}
Meyer, Y.: Oscillating patterns in image processing and nonlinear evolution equations, volume 22 of university lecture series.
\newblock Amer. Math. Soc. \textbf{364} (2001)

\bibitem{L2017Measuring}
Métivier, L., Brossier, R., Mérigot, Q., Oudet, E., Virieux, J.: Measuring the misfit between seismograms using an optimal transport distance: application to full waveform inversion.
\newblock Geophys. J. Int. \textbf{205}(1), 332--364 (2017)

\bibitem{6459598}
Ng, M.K., Yuan, X., Zhang, W.: Coupled variational image decomposition and restoration model for blurred cartoon-plus-texture images with missing pixels.
\newblock IEEE Trans. Image Process. \textbf{22}(6), 2233--2246 (2013).
\newblock \doi{10.1109/TIP.2013.2246520}

\bibitem{osher2003image}
Osher, S., Sol{\'e}, A., Vese, L.: Image decomposition and restoration using total variation minimization and the $h^{-1}$ norm.
\newblock Multiscale Model. Simul. \textbf{1}(3), 349--370 (2003)

\bibitem{2015Optimal}
Papadakis, N.: Optimal transport for image processing.
\newblock Ph.D. thesis, Universit{\'e} de Bordeaux; Habilitation thesis (2015)

\bibitem{Peyr2019Computational}
Peyr{\'e}, G., Cuturi, M., et~al.: Computational optimal transport: With applications to data science.
\newblock Found. Trends Mach. Learn. \textbf{11}(5-6), 355--607 (2019)

\bibitem{rachev2006mass}
Rachev, S.T., R{\"u}schendorf, L.: Mass Transportation Problems: Volume 1: Theory.
\newblock Springer Science \& Business Media (2006)

\bibitem{2000The}
Rubner, Y., Tomasi, C., Guibas, L.J.: The earth mover's distance as a metric for image retrieval.
\newblock Int. J. Comput. Vision \textbf{40}, 99--121 (2000)

\bibitem{1992Nonlinear}
Rudin, L.I., Osher, S., Fatemi, E.: Nonlinear total variation based noise removal algorithms.
\newblock Phys. D \textbf{60}(1-4), 259--268 (1992)

\bibitem{2017Vector}
Ryu, E.K., Chen, Y., Li, W., Osher, S.: Vector and matrix optimal mass transport: theory, algorithm, and applications.
\newblock SIAM J. Sci. Comput. \textbf{40}(5), A3675--A3698 (2018)

\bibitem{Chan2002EULER}
Shen, J., Kang, S.H., Chan, T.F.: Euler's elastica and curvature-based inpainting.
\newblock SIAM J. Appl. Math. \textbf{63}(2), 564--592 (2003)

\bibitem{2015Convolutional}
Solomon, J.M., Goes, F.D., Peyr, G., Cuturi, M., Butscher, A., Nguyen, A., Du, T., Guibas, L.J.: Convolutional wasserstein distances: efficient optimal transportation on geometric domains.
\newblock ACM Trans. Graphic. \textbf{34}(4), 1--11 (2015)

\bibitem{2003StaImage}
Starck, J.L., Elad, M., Donoho, D.L.: Image decomposition: Separation of texture from piecewise smooth content.
\newblock Proc. SPIE Int. Soc. Opt. Eng. \textbf{5207}(2) (2003)

\bibitem{Euler}
Tai, X.C., Hahn, J., Chung, G.J.: A fast algorithm for euler's elastica model using augmented lagrangian method.
\newblock SIAM J. Imaging Sci. \textbf{4}(1), 313--344 (2011)

\bibitem{2001A}
Vese, L.: A study in the bv space of a denoising-deblurring variational problem.
\newblock Applied Mathematics and Optimization \textbf{44}(2), 131--161 (2001)

\bibitem{vese2003modeling}
Vese, L.A., Osher, S.J.: Modeling textures with total variation minimization and oscillating patterns in image processing.
\newblock J. Sci. Comput. \textbf{19}, 553--572 (2003)

\bibitem{2007A}
Wang, Y., Yin, W., Zhang, Y.: A fast algorithm for image deblurring with total variation regularization.
\newblock Rice University CAAM Technical Report TR07-10 pp. 1--19 (2007)

\bibitem{2010Augmented}
Wu, C., Tai, X.C.: Augmented lagrangian method, dual methods, and split bregman iteration for rof, vectorial tv, and high order models.
\newblock SIAM J. Imaging Sci. \textbf{3}(3), 300--339 (2010)

\bibitem{2017Augmented}
Wu, C., Zhang, J., Tai, X.C.: Augmented lagrangian method for total variation restoration with non-quadratic fidelity.
\newblock Inverse Probl. Imaging \textbf{5}(1), 237--261 (2017)

\bibitem{pmlr-v115-xie20b}
Xie, Y., Wang, X., Wang, R., Zha, H.: A fast proximal point method for computing exact wasserstein distance.
\newblock In: R.P. Adams, V.~Gogate (eds.) Proceedings of The 35th Uncertainty in Artificial Intelligence Conference, \emph{Proceedings of Machine Learning Research}, vol. 115, pp. 433--453. PMLR (2020).
\newblock \urlprefix\url{https://proceedings.mlr.press/v115/xie20b.html}

\bibitem{xie2024randomized}
Xie, Y., Wang, Z., Zhang, Z.: Randomized methods for computing optimal transport without regularization and their convergence analysis.
\newblock J. Sci. Comput. \textbf{100}(2), 37 (2024).
\newblock \doi{https://doi.org/10.1007/s10915-024-02570-w}

\bibitem{2023A}
Yang, W., Huang, Z., Zhu, W.: A first-order rician denoising and deblurring model.
\newblock Inverse Probl. Imaging \textbf{17}(6), 1139--1164 (2023)

\bibitem{bregman}
Yin, W., Osher, S., Goldfarb, D., Darbon, J.: Bregman iterative algorithms for $\ell\_1$-minimization with applications to compressed sensing.
\newblock SIAM J. Imaging Sci. \textbf{1}(1), 143--168 (2008).
\newblock \doi{10.1137/070703983}

\bibitem{2018On}
Zeng, C., Wu, C.: On the edge recovery property of noncovex nonsmooth regularization in image restoration.
\newblock SIAM J. Numer. Anal. \textbf{56}(2), 1168--1182 (2018)

\bibitem{2020zhuLp}
Zhu, W.: Image denoising using $l^p$-norm of mean curvature of image surface.
\newblock J. Sci. Comput. \textbf{83}(2), 32 (2020)

\bibitem{zhu}
Zhu, W.: A first-order image restoration model that promotes image contrast preservation.
\newblock J. Sci. Comput. \textbf{88}, 46 (2021)

\bibitem{zhu2012image}
Zhu, W., Chan, T.: Image denoising using mean curvature of image surface.
\newblock SIAM J. Imaging Sci. \textbf{5}(1), 1--32 (2012)

\end{thebibliography}

%
%

\end{document}